\documentclass[11pt,reqno]{amsart}

\usepackage[english]{babel}
\usepackage{amsmath, amsthm, amssymb}
\usepackage{amsmath}
\usepackage{amsfonts}
\usepackage{amsthm}
\usepackage{mathrsfs}
\usepackage{amsrefs}
\usepackage{hyperref}
\usepackage{xcolor}
\usepackage{enumitem}
\usepackage{dsfont}

\theoremstyle{plain}
\newtheorem{theorem}{Theorem}[section]
\newtheorem{corollary}[theorem]{Corollary}
\newtheorem{lemma}[theorem]{Lemma}
\newtheorem{proposition}[theorem]{Proposition}
\newtheorem*{theorem*}{Claim}

\theoremstyle{definition}

\newtheorem{remark}[theorem]{Remark}

\numberwithin{equation}{section}

\newcommand{\hcal}{\mathcal{H}}
\newcommand{\R}{\mathbb{R}}

\newcommand{\jacobi}{J_u}
\newcommand{\vv}{b}
\newcommand{\bb}{\widehat{b}}

\renewcommand{\div}{{\rm div}}
\newcommand{\tr}{{\rm tr}}
\renewcommand{\d}{\mathop{}\!\mathrm{d}}
\newcommand{\indicator}{\mathds{1}}
\newcommand{\supp}{\, {\rm supp} \,}
\newcommand{\elliptic}{c_0}
\newcommand{\bounded}{C_0}
\newcommand{\cc}{\textrm{\bf c}}
\newcommand{\anew}{\mathcal{A}}
\newcommand{\azero}{\mathcal{A}_{0}}
\newcommand{\nn}{\textrm{\bf n}}
\newcommand{\NN}{\textrm{\bf N}}

\newcounter{desccount}

\newcommand{\descref}[1]{\hyperref[#1]{#1}}

\begin{document}
	
\title[Energy estimate up to the boundary for stable solutions]{
Energy estimate up to the boundary for stable solutions to semilinear elliptic problems
}
\author{I\~{n}igo U. Erneta}
\address{I.~U. Erneta\textsuperscript{1,2} ---
\textsuperscript{1}Centre de Recerca Matem\`{a}tica, Edifici C, Campus Bellaterra, 08193 Bellaterra, Spain \&
\textsuperscript{2}Universitat Polit\`{e}cnica de Catalunya, Departament de Matem\`{a}tiques, Diagonal 647, 08028 Barcelona, Spain }
\email{inigo.urtiaga@upc.edu}

\thanks{The author acknowledges financial support from MINECO grant MDM-2014-0445-18-1 through the Mar\'{i}a de Maeztu Program for Units of Excellence in R\&D.
He is additionally supported by Spanish grants MTM2017-84214-C2-1-P and PID2021-123903NB-I00 funded by MCIN/AEI/10.13039/501100011033 and by ERDF ``A way of making Europe''.
The author is also supported by Catalan project 2021 SGR 00087.
}

\maketitle

\begin{abstract}

We obtain a universal energy estimate up to the boundary for stable solutions of semilinear equations with variable coefficients.
Namely, we consider solutions to $- L u = f(u)$, where $L$ is a linear uniformly elliptic operator
and $f$ is $C^1$,
such that the linearized equation $-L - f'(u)$ has nonnegative principal eigenvalue.
Our main result is an estimate for the $L^{2+\gamma}$ norm of the gradient of stable solutions 
vanishing on the flat part of a half-ball,
for \emph{any} nonnegative and nondecreasing $f$.
This bound only requires the elliptic coefficients to be Lipschitz.
As a consequence, our estimate continues to hold in general $C^{1,1}$ domains if we further assume the nonlinearity $f$ to be convex.
This result is new even for the Laplacian, for which a $C^3$ regularity assumption on the domain was needed.
%
%
%
\end{abstract}


\section{Introduction}

Given a bounded domain $\Omega \subset \R^n$ and a function $f \in C^1(\R)$,
we consider stable solutions $u\colon \overline{\Omega} \to \R$
to the semilinear boundary value problem
\begin{equation}\label{pb:omega}
\left\{
\begin{array}{cl}
- L u = f(u) & \text{ in } \Omega\\
u = 0 & \text{ on } \partial \Omega.
\end{array}
\right.
\end{equation}
Throughout the text, 
$L$ denotes a uniformly elliptic operator of the form
\begin{equation}
\label{def:op}
L = a_{ij}(x) \partial_{ij} + \vv_i(x) \partial_i,
\quad a_{ij}(x) = a_{ji}(x).
\end{equation}
A solution $u$ of \eqref{pb:omega} is called \emph{stable} if the principal eigenvalue 
(with respect to Dirichlet conditions)
of the linearized equation $\jacobi := L + f'(u)$
 is nonnegative.\footnote{Here we adopt the sign convention $\jacobi \varphi = - \mu \varphi$ for the eigenvalues $\mu$ of $\jacobi$.}
When the problem is variational, this amounts to the nonnegativity of the second variation, a necessary condition for the minimality of $u$. 

The goal of the present article is to obtain a universal energy estimate for stable solutions to \eqref{pb:omega}
in the spirit of the pioneering work of Cabr\'{e}, Figalli, Ros-Oton, and Serra~\cite{CabreFigalliRosSerra} 
for the Laplacian.
In~\cite{CabreFigalliRosSerra}, the authors proved two types of a priori bounds for 
classical stable solutions
when $L = \Delta$.
Namely,
a control of the $L^{2+\gamma}$ norm of the gradient 
(for some $\gamma > 0$) by the $L^1$ norm of the function, valid in all dimensions,
and an estimate of the 
H\"{o}lder
norm of the solution when $n \leq 9$.
The latter result is optimal,
since there are examples of singular (unbounded) stable solutions in dimensions $n \geq 10$.
A notable feature of these estimates is that they do not depend on the nonlinearity, 
which is assumed to be nonnegative, nondecreasing, and convex.
Thanks to this,
the paper \cite{CabreFigalliRosSerra} answered positively two long-standing open questions of 
Brezis and V\'{a}zquez~\cite{BrezisVazquez} and of Brezis~\cite{Brezis-IFT}
concerning the regularity of extremal solutions (which are $L^1$ limits of classical stable solutions), recalled briefly below.

Here we will be interested in extending the $L^{2+\gamma}$ energy estimate to operators with variable coefficients as in \eqref{def:op}.
Our main achievement is to make 
the constants in our bounds depend on the $C^{0,1}$ norm of $a_{ij}$ and the $L^{\infty}$ norm of $\vv_i$,
this being the major difficulty in our proofs.
As a consequence, we will obtain 
a global estimate in $C^{1,1}$ domains.
This result is new even when $L$ is the Laplacian,
as~\cite{CabreFigalliRosSerra} required a $C^3$ regularity assumption on the domain.
For this, starting from a curved boundary, we flatten it out locally by a change of variables.
In the new coordinates, our solution is still a stable solution to an equation of the form~\eqref{pb:omega},
where the new operator $L$ now involves the derivatives of 
the flattening map.
More precisely, the new coefficients $a_{ij}$ depend on the differential of this map, 
while $\vv_i$ additionally depend on its Hessian.
It follows that the $C^{0,1}$ and $L^{\infty}$ regularity of the coefficients corresponds to a $C^{1,1}$ domain.
In particular, it will suffice to prove a priori estimates in half-balls with the stated dependence on the coefficients.

Furthermore, when $n \leq 9$, our energy bound (as well as the auxiliary Hessian estimates in Theorem~\ref{thm:sz} below) will be crucial to establish H\"{o}lder estimates up to the boundary in $C^{1,1}$ domains.
We will tackle this issue in our forthcoming paper~\cite{ErnetaBdy2}, where we extend the optimal $C^{\alpha}$ bounds of \cite{CabreFigalliRosSerra} to equations with coefficients.
The previous work~\cite{CabreFigalliRosSerra} relied on delicate contradiction-compactness arguments which do not allow to quantify the constants in the estimates.
Here, thanks to  a new device of Cabr\'{e}~\cite{CabreRadial} for the Laplacian in flat domains we will be able to give a direct, quantitative proof of all our estimates in~\cite{ErnetaBdy2}.


The study of the regularity of stable solutions to \eqref{pb:omega} was originally motivated by problems in combustion theory.
In that setting, the interest lies in positive, nondecreasing, convex, and superlinear nonlinearities $f$ accounting for the reaction of a combustible mixture.
It is also natural to consider a multiple $\lambda f$ of the nonlinearity,
where $\lambda > 0$ is a nondimensional parameter measuring the relative strength of the reaction with respect to the processes modeled by $L$.
Applying the implicit function theorem at $\lambda = 0$ and by the properties of $f$, one obtains 
a
branch of positive classical stable solutions $\{u_{\lambda} \}_{0 < \lambda < \lambda^{\star}}$
of 
$-L u_{\lambda} = \lambda f(u_{\lambda})$ in $\Omega$, $u_{\lambda} = 0$ on $\partial \Omega$,
where $0 < \lambda^{\star} < \infty$ is the maximal threshold for the existence of classical solutions to this problem.
Moreover, by maximum principle, $\lambda \mapsto u_{\lambda}$ is increasing in $(0, \lambda^{\star})$ and converges in $L^1$ to a 
weak (distributional) 
solution $u^{\star}$, the so called \emph{extremal solution}; see, for instance~\cite{CrandallRabinowitz, Dupaigne, Brezis-IFT}.

By construction, the extremal solution $u^{\star}$
is a priori only in $L^1$ and can be unbounded.
In~\cite{BrezisVazquez}, Brezis and V\'{a}zquez gave a characterization of singular (unbounded) extremal solutions in the energy space $W^{1,2}_{0}(\Omega)$ 
when $L$ is the Laplacian.
Their result led them to ask whether extremal solutions are necessarily in this space;
see~\cite{BrezisVazquez}*{Problem 1}.
This question has been addressed in various works,
always considering the model operator $L = \Delta$.
The first result in this direction was obtained by Nedev~\cite{Nedev},
who showed the validity of the claim for $n \leq 5$.
Later, 
assuming $\Omega$ to be convex
(or, more generally, ``bean shaped''), he was able to extend this result to all dimensions in an unpublished preprint~\cite{Nedev2}
(which is recalled and proven again in~\cite{CabreSanchon}).
Then, Cabr\'{e} and Capella studied radial stable solutions in $\Omega = B_1$, 
showing that $u^{\star} \in W^{3, 2}(B_1)$ in this case.
After that, Cabr\'{e} and Ros-Oton~\cite{CabreRosOton-DoubleRev} proved the claim for $n \leq 6$ in domains of double revolution,
and Villegas~\cite{Villegas} obtained the same result in general smooth domains.
Recently, Cabr\'{e}, Figalli, Ros-Oton, and Serra~\cite{CabreFigalliRosSerra}
settled the conjecture, showing that $u^{\star} \in W^{1,2+\gamma}_{0}(\Omega)$ in all dimensions,
where $\gamma > 0$ depends only on $n$, and $\Omega$ is of $C^3$ class.
For this, as mentioned above, they proved a universal energy estimate for smooth stable solutions.
Then, they applied it to the functions $\{u_{\lambda}\}_{0 < \lambda < \lambda^{\star}}$ and passed to the limit as $\lambda \to \lambda^{\star}$.

For further regularity properties of $u^{\star}$, the dimension of the space plays a critical role.
Notice that, by the linear theory, the smoothness of $u^{\star}$ follows from its boundedness.
When $n \geq 10$, explicit unbounded extremal solutions had been known for a long time, while no such examples were found in lower dimensions.
In~\cite[Open problem 1]{Brezis-IFT}, Brezis asked whether the extremal solution was always bounded in the latter case.
His question prompted a series of works trying to establish $L^{\infty}$ bounds for classical stable solutions in the range $n \leq 9$.
Recently, in the breakthrough paper~\cite{CabreFigalliRosSerra} mentioned above, the question was answered positively for the Laplacian in $C^3$ domains.
For more information on that problem, see the references in~\cite{CabreFigalliRosSerra} or, for instance,~\cite{ErnetaBdy2}.

\subsection{Main results}
We assume that the coefficient matrix $A(x) = (a_{ij}(x))$ is uniformly elliptic in $\Omega$, that is, there are positive constants $\elliptic$, $\bounded$ such that 
\begin{equation}
\label{elliptic}
\elliptic |p|^2 \leq a_{ij}(x) p_i p_j \leq \bounded |p|^2 \quad \text{ for all } p \in \R^n.
\end{equation}
Our global theorem in $C^{1,1}$ domains requires the assumption
\begin{equation}
\label{reg:coeffs}
a_{ij} \in C^{0,1}(\overline{\Omega}), \quad \vv_i \in L^{\infty}(\Omega) \cap C^0(\Omega),
\end{equation}
For our local results in half-balls, we further need the auxiliary condition
\begin{equation}
\label{reg:b}
\vv_i \in C^{0}(\overline{\Omega}).
\end{equation}
We will be able to remove \eqref{reg:b} by an approximation argument, 
as explained in Remark~\ref{remark:approximation}.

Since we always assume
$a_{ij} \in C^{0,1}(\overline{\Omega}) = W^{1, \infty}(\Omega)$, 
we can write $L$ in divergence form
\begin{equation}
\label{op:nondiv}
L u = \div \left( A(x) \nabla u\right) + \bb(x) \cdot \nabla u,
\end{equation}
where $\bb(x) = (\bb_{i}(x))$ is the vector field given by 
\begin{equation}
\label{def:bvector}
\bb_i(x) = \vv_i(x) - \partial_k a_{ki}(x). 
\end{equation}
Notice that $\bb_i$ is in $L^{\infty}(\Omega)$
by assumption \eqref{reg:coeffs}.

Having specified the regularity of the coefficients,
we can give a more precise definition of stable solution.
Assuming \eqref{elliptic} and \eqref{reg:coeffs}, 
we consider the class of \emph{strong solutions} to~\eqref{pb:omega},
that is,
functions $u \in C^{0}(\overline{\Omega}) \cap W^{2,n}_{\rm loc}(\Omega)$ such that 
$- L u = f(u)$ a.e. in $\Omega$ and $u = 0$ on~$\partial \Omega$.
As commented above, a strong solution $u$ of 
\eqref{pb:omega} 
is \emph{stable} if the principal eigenvalue of the linearized equation at $u$ is nonnegative.
Equivalently (see~\cite{BerestyckiNirenbergVaradhan}),
the solution $u$ is stable if there exists a function $\varphi \in W^{2,n}_{\rm loc}(\Omega)$ such that
\begin{equation}
\label{stable:point}
\left\{
\begin{array}{cl}
\jacobi \varphi \leq 0 &\text{ a.e. in } \Omega,\\
\varphi > 0 &\text{ in } \Omega,
\end{array}
\right.
\end{equation}
where, recall, $\jacobi = L + f'(u)$ denotes the Jacobi operator (the linearization) at $u$.
We would like to point out
that the notion of stable solution refers only to the equation satisfied by $u$ and not to its boundary value.

Our energy estimate in $C^{1,1}$ domains will apply to strong stable solutions as above.
In a sense, 
these functions
are the natural replacement of the classical solutions 
for the Laplacian in $C^3$ domains
considered in~\cite{CabreFigalliRosSerra}.
Notice that, 
since $f(u) \in L^{\infty}(\Omega)$,
by $L^p$ estimates in $C^{1,1}$ domains (see~\cite{GilbargTrudinger}*{Theorem~9.13}),
strong solutions belong to $W^{2,p}(\Omega)$ for all $p < \infty$.
For further regularity properties, more assumptions on the coefficients and the domain are needed.
In fact, our a priori estimates in half-balls below require our solutions to have third weak derivatives,
but we will be able to remove this assumption by an approximation argument; see Remark~\ref{remark:approximation}.

We now state the main result of this work, an energy estimate up to the boundary in flat domains.
For $\rho > 0$, we denote 
the half-ball of radius $\rho$ centered at $0$ by 
\[
B_{\rho}^{+} := \{ x_n > 0 \} \cap B_{\rho}, 
\]
where $B_{\rho} = \{|x| < \rho\} \subset \R^n$ is the full-ball.
We also write 
\[
\partial^{0} B_{\rho}^{+} = \{x_n = 0\} \cap \partial B_{\rho}^+.
\]
for the lower boundary of $B_{\rho}^{+}$.
In the results below, $C = C(\ldots)$ denotes a constant $C$ depending only on the quantities appearing inside the parentheses.
We have the following:

\begin{theorem}
\label{thm:holder}
Let $L$ satisfy conditions~\eqref{elliptic},~\eqref{reg:coeffs}, and~\eqref{reg:b} in $\Omega = B_1^{+} \subset \R^n$.
Assume that $f \in C^1(\R)$ is nonnegative and nondecreasing.

Let $u \in W^{3,p}(B_1^{+})$, for some $p > n$, be a nonnegative stable solution to $- L u = f(u)$ in $B_1^{+}$, with $u = 0$ on $\partial^0 B_1^{+}$.

Then 
\[
\|\nabla u\|_{L^{2+ \gamma}(B_{1/2}^{+})} \leq C \|u\|_{L^1(B_1^{+})},
\]
where $\gamma = \gamma(n) > 0$ and 
$C = C(n, \elliptic, \bounded, \|\nabla a_{ij}\|_{L^{\infty}(B_{1}^{+})}, \|\vv_i \|_{L^{\infty}(B^{+}_1)})$.
\end{theorem}

\begin{remark}
\label{remark:regul}
Note that we are further assuming $u \in W^{3,p}(B_1^{+})$ with $p > n$. 
In particular, by Sobolev embedding, 
$u$ is in $C^2(\overline{B_1^+})$
and the solution is classical.\footnote{For the embedding in half-balls, just apply the usual Sobolev embedding in the full ball to a third order reflection of $u$, for instance, letting \[ \textstyle u(x', x_n) = -10 u(x',-x_n) + 160 u(x',-\frac{x_n}{2}) - 405 u(x',-\frac{x_n}{3}) + 256 u(x',-\frac{x_n}{4}) \] for $x_n < 0$ and $x' \in \R^{n-1}$, which is in $W^{3,p}(B_1) \subset C^2(\overline{B_1})$.}
By approximation, the estimate will also hold for strong solutions (see Theorem~\ref{thm:c11} and Remark~\ref{remark:approximation} below).
We need third weak derivatives in order to have a cancellation which removes the nonlinearity in the stability condition.
This step is crucial for our bounds to be independent of $f$.
\end{remark}

\begin{remark}
\label{remark:cont}
The continuity $\vv_i \in C^0(\overline{B_1^+})$ up to the boundary (assumption \eqref{reg:b} above) will allow us to control these coefficients on certain surface integrals over $\partial^0 B_1^+$ arising in the proof.
Assuming only $\vv_i \in L^{\infty}(B_1^+)$ does not suffice for such estimates on surfaces.
\end{remark}

To prove Theorem~\ref{thm:holder}, 
the stability condition \eqref{stable:point}
will come into play
 through a useful integral inequality 
that has already appeared in our previous work~\cite{ErnetaInterior}.
Since the coefficient matrix $A(x) = (a_{ij}(x))$ is positive definite, 
it gives rise to a norm
\[ |p|_{A(x)} := \left( a_{ij}(x) p_i p_j\right)^{1/2} \quad \text{ for } p \in \R^n. \]
In~\cite{ErnetaInterior}, we showed that if $u$ is stable, 
then 
\begin{equation}
\label{ineq:stable}
\int_{\Omega} f'(u) \xi^2 \d x 
\leq \int_{\Omega} 
\left|\nabla \xi-{\textstyle\frac{1}{2}} \xi A^{-1}(x) \bb(x)\right|^2_{A(x)}
\d x 
\quad \text{ for all } \xi \in C^\infty_c(\Omega),
\end{equation}
where $\bb(x)$ is the vector field introduced in \eqref{def:bvector} above.
Essentially, \eqref{ineq:stable} 
follows from the inequality in \eqref{stable:point} multiplying by $\xi^2/ \varphi$, integrating by parts, and completing squares.
We often refer to \eqref{ineq:stable} as the ``integral stability inequality'' to distinguish it from the pointwise condition \eqref{stable:point} above.
Moreover, we would like to point out that the inequality~\eqref{ineq:stable} is not equivalent to our stability condition~\eqref{stable:point} in general; see~\cite{ErnetaInterior}.

A fundamental ingredient in the proof of Theorem~\ref{thm:holder} 
will be to control the Hessian of a stable solution in half-balls.
The following boundary Hessian estimates 
can be interpreted as a generalization of a geometric stability condition due to Sternberg and Zumbrun~\cite{SternbergZumbrun1}.
Throughout the paper,
a constant depending only on $n$, $\elliptic$, and $\bounded$ will be called \emph{universal}.

\begin{theorem}
\label{thm:sz}
Let $u \in W^{3,p}(B_1^{+})$, for some $p > n$,
be a nonnegative stable solution of 
$- L u = f(u)$ in $B_{1}^{+}$, with $u = 0$ on $\partial^0 B_{1}^{+}$.
Assume that $f \in C^1(\R)$ is nonnegative.
Assume that $L$ satisfies conditions~\eqref{elliptic},~\eqref{reg:coeffs}, and~\eqref{reg:b} in $\Omega = B_1^{+}$, and that
\[
\|D A\|_{L^{\infty}(B_{1}^{+})} + \|\vv\|_{L^{\infty}(B_{1}^{+})} \leq \varepsilon
\]
for some $\varepsilon > 0$.

Then
\begin{equation}
\label{eq:szz}
\begin{split}
\int_{B_1^{+}}  \anew^2 \eta^2 \d x 
&\leq C \int_{B_{1}^{+}} |\nabla u|^2 \left( |\nabla \eta|^2 + |D^2 (\eta^2)|+ \varepsilon |\nabla (\eta^2)| + \varepsilon^2 \eta^2\right) \d x\\
&\quad \quad \quad + C\int_{B_1^{+}} |D^2 u | |\nabla u| \big(|\nabla(\eta^2)| + \varepsilon \eta^2\big) \d x \\
& \quad \quad \quad \quad \quad + C \int_{\partial^0 B_1^{+}} |\nabla u|^2 (|\nabla (\eta^2)|  +\varepsilon\eta^2) \d \hcal^{n-1}
\end{split}
\end{equation}
for all $\eta \in C^{\infty}_c(B_1)$,
where $C$ is a universal constant and
\begin{equation}
\label{def:a}
\anew := 
\left\{
\begin{array}{ll}
\Big(\tr(A(x) D^2 u A(0) D^2 u) -|\nabla u|_{A(0)}^{-2}  |D^2 u A(0) \nabla u|^2_{A(x)} \Big)^{1/2} & \text{ if } \nabla u \neq 0 \\
0 & \text{ if } \nabla u = 0.
\end{array}
\right.
\end{equation}

Assume moreover that $f$ is nondecreasing and $\varepsilon \leq \varepsilon_0$.
Then
\begin{equation}
\label{ineq:pohozaev}
\|\nabla u\|_{L^2(\partial^0 B_{2/3}^{+})} \leq C \|\nabla u\|_{L^2(B_1^{+})},
\end{equation}
\begin{equation}
\label{ineq:hessgrad}
\||\nabla u| D^2 u\|_{L^1(B_{4/7}^{+})} \leq C \|\nabla u\|_{L^2(B_1^{+})}^2,
\end{equation}
\begin{equation}
\label{ineq:szkinda}
\| \anew \|_{L^2(B_{1/2}^{+})} \leq C \|\nabla u\|_{L^2(B_1^{+})},
\end{equation}
and
\begin{equation}
\label{ineq:hessian}
\| D^2 u \|_{L^1(B_{4/7}^{+})} \leq C \|\nabla u\|_{L^2(B_1^{+})},
\end{equation}
where $\varepsilon_0 > 0$ and $C$ are universal constants.
\end{theorem}

%

To prove the first bound~\eqref{eq:szz} in Theorem~\ref{thm:sz},
we will exploit the integral stability inequality~\eqref{ineq:stable} by choosing appropriate test functions.
Letting $\xi = \cc \eta$ in~\eqref{ineq:stable} with $\Omega = B_1^+$,
where $\cc$, $\eta$ are smooth functions satisfying $\cc = 0$ on $\partial^{0} B_{1}^{+}$ and $\supp \eta \subset B_1$,
if we integrate by parts, then \eqref{ineq:stable} becomes
\begin{equation}
\label{stab:half:jac}
\int_{B_{1}^{+}}  \cc \jacobi \cc \, \eta^2\d x 
\leq \int_{B_{1}^{+}} 
\cc^2 \left|\nabla \eta-{\textstyle\frac{1}{2}} \eta A^{-1}(x) b(x)\right|^2_{A(x)}
\d x.
\end{equation}

In order to obtain universal estimates, the crucial point will be to choose $\cc$ in such a way that the 
Jacobi operator
$\jacobi \cc$ 
in the left-hand side of \eqref{stab:half:jac}
becomes independent of the nonlinearity.
Thus, in the proof of \eqref{eq:szz},
our choice will be a smooth approximation of
\[
\cc(x) = |\nabla u(x)|_{A(0)} - \NN \cdot \nabla u (x)
\]
for an appropriate constant vector field $\NN \colon \R^n_+ \to \R^n$ 
(given by~\eqref{def:n:field} in Section~\ref{section:hessian} below).
Here, we need $f \geq 0$ to make sure that such an approximation of $\cc$ vanishes on $\partial^0 B_1^+$,
but otherwise is a technical assumption in this step.

Under a smallness assumption on the coefficients ($\varepsilon \leq \varepsilon_0$),
the function $\anew$ in~\eqref{def:a} controls part of the Hessian of $u$ (as explained in~\cite{ErnetaInterior} or in Section~\ref{section:hessian} below).
We can further bound the full Hessian by assuming that the equation has a sign $- L u = f(u) \geq 0$.
For the the final form of the Hessian estimates in~\eqref{ineq:hessgrad},~\eqref{ineq:szkinda}, and~\eqref{ineq:hessian},
we need to control the 
third term
in the right-hand side of~\eqref{eq:szz},
which is a surface integral and arises at every integration by parts.
To control such an integral requires both the monotonicity of $f$
and the stability of $u$,
while the previous works~\cite{CabreFigalliRosSerra, CabreQuant} only needed the condition on $f$.
The reason for this is an additional Hessian error which does not appear 
for the Laplacian on $C^3$ domains when trying to control the boundary integral.

Once Theorem~\ref{thm:sz} is available, 
our main result, Theorem~\ref{thm:holder}, 
will follow directly by the ideas of~\cite{CabreFigalliRosSerra, CabreRadial} combined with a scaling and covering argument.

To conclude this section, we state our energy estimate in general domains of $C^{1,1}$ class.
Approximating \eqref{pb:omega} by 
stable solutions to
smoother problems 
(as explained next in Remark~\ref{remark:approximation}),
flattening the boundary,
the result will follow easily from Theorem~\ref{thm:holder}
and 
by the interior estimates obtained in our previous work~\cite{ErnetaInterior}.
This argument requires the convexity of $f$ to ensure that the approximating sequence of stable solutions converges to the original one.
The same procedure can be used to obtain H\"{o}lder estimates up to the boundary in $C^{1,1}$ domains, 
which has been carried out in our forthcoming work~\cite{ErnetaBdy2}.
Since the ideas in both papers are very similar,
we defer the complete proof of Theorem~\ref{thm:c11} below to~\cite{ErnetaBdy2},
where we implement the approximation and flattening argument in full detail.
Here, we just give indications in Remark~\ref{remark:approximation}, after the theorem.

\begin{theorem}[\cite{ErnetaBdy2}]
\label{thm:c11}
Let $\Omega \subset \R^n$ be a bounded domain of class $C^{1,1}$
and let $L$ satisfy conditions~\eqref{elliptic} and~\eqref{reg:coeffs} in $\Omega$.
Assume that $f \in C^1(\R)$ is nonnegative, nondecreasing, and convex.

Let $u \in C^{0}(\overline{\Omega}) \cap W^{2,n}_{\rm loc}(\Omega)$ be a nonnegative stable solution of $- L u = f(u)$ in $\Omega$, with $u = 0$ on $\partial \Omega$.

Then
\[
\|\nabla u\|_{L^{2+\gamma}(\Omega)} \leq C \|u\|_{L^1(\Omega)},
\]
where $\gamma = \gamma(n) > 0$ and 
$C = C(\Omega, n, \elliptic, \bounded, \| \nabla a_{ij}\|_{L^{\infty}(\Omega)}, \|b_i\|_{L^{\infty}(\Omega)})$.
\end{theorem}

\begin{remark}
\label{remark:approximation}
As mentioned above, our energy estimate in $C^{1,1}$ domains 
will follow from 
Theorem~\ref{thm:holder}
but, unlike this result,
it does not require third derivatives of the solution
or assumption~\eqref{reg:b} (the continuity of $\vv_i$ up to the boundary).
To achieve this, we consider an exhaustion of $\Omega$ by smooth sets $\Omega_k$.
Using $u$ as a barrier, by monotone iteration,
 we construct strong stable solutions $u_k$ to a semilinear equation 
$-L_k u_k = f_k(u_k)$ in $\Omega_k$ with smoother coefficients.

Flattening the boundary $\partial \Omega_k$, we obtain solutions in the half-ball, 
where we would like to apply Theorem~\ref{thm:holder}.
For this, we need to ensure the existence of third weak derivatives in $L^p$ for these solutions,
which is guaranteed if the new coefficients $(\vv^{k}_i)_{k}$ are sufficiently regular.\footnote{For instance, suppose that $a_{ij} \in C^{0,1}(\overline{B_1^+})$ and $\vv_i \in W^{1,p}(B_1^+)$ for some $p > n$, and let $u \in W^{2,p}(B_1^+)$ be a strong solution to $-L u = f(u)$ in $B_1^+$, $u = 0$ on $\partial^0 B_1^+$. Since $f(u) \in L^{\infty}(B_1^+)$, by Calder\'{o}n-Zygmund estimates (see~\cite{GilbargTrudinger}*{Theorem~9.13}) we have $u \in W_{\rm loc}^{2, q}(B_1^{+} \cup \partial^0 B_1^{+})$ for all $q <\infty$. Formally taking tangential derivatives, for $k = 1, \ldots, n-1$ we obtain $- L u_k = f'(u) u_k + \partial_k a_{ij}(x) u_{ij} + \partial_k \vv_i(x) u_i \in L^p_{\rm loc}(B_1^{+} \cup \partial^0 B_1^{+})$ and $u_k = 0$ on $\partial^0 B_1^+$, hence, again by $L^p$ estimates, we deduce $u \in W^{3, p}_{\rm loc}(B_1^{+} \cup \partial^0 B_1^{+})$. It remains to check that the weak derivative $u_{nnn}$ exists and lies in $L^{p}_{\rm loc}(B_1^{+} \cup \partial^0 B_1^{+})$, but this follows easily from the equation.}
The interior continuity of $\vv_i$ (assumption~\eqref{reg:coeffs} above) will make sure that $\vv_i^{k} \to \vv_i$ locally uniformly in $\Omega$,
which is essential to show that $u$ is a barrier.

Finally, we need the convexity of $f$ for $u_k$ to converge to the original solution $u$ and not to some other limit.
The deeper reason behind this is that stable solutions with convex nonlinearities 
are unique; see~\cite{Dupaigne, ErnetaBdy2}.
For $C^3$ domains and smooth coefficients, we do not need the approximation procedure
and we could give the analogue of Theorem~\ref{thm:c11} without the convexity assumption on $f$.
\end{remark}

\subsection{Outline of the article}
Section~\ref{section:hessian} is devoted to the proof of Theorem~\ref{thm:sz} containing the Hessian estimates up to the boundary.
In Section~\ref{section:higher:int} we prove 
Theorem~\ref{thm:holder},
the energy estimate in half-balls.

In Appendix~\ref{app:interpolation} we recall some useful interpolation inequalities of Cabr\'{e}~\cite{CabreRadial}.
Finally, in Appendix~\ref{app:simon} we recall Simon's lemma~\cite{Simon} for absorbing errors in larger balls.

\section{Boundary Hessian estimates}
\label{section:hessian}

Recall the function
$\anew \colon \overline{B}_1 \to \R$
introduced in~\eqref{def:a} in the statement of Theorem~\ref{thm:sz}.
It can also be written as
\begin{equation}
\label{def:anew}
\anew =
\Big( \|A^{1/2}(x)D^2 u A^{1/2}(0)\|_{\rm HS}^2 
- 
|A^{1/2}(x) D^2 u A^{1/2}(0) \nn(x)|^2
\Big)^{1/2} \quad  \text{in } \{\nabla u \neq 0\},
\end{equation}
where $\|\cdot \|_{\rm HS}$ denotes the Euclidean Hilbert-Schmidt norm for matrices
and $\nn(x)$ is the unit vector field
$\nn \colon (B_1^{+} \cup \partial^0 B_1^{+}) \cap \{ \nabla u \neq 0\} \to \R$
given by
\begin{equation}
\label{n:vector}
\nn(x) := |\nabla u|^{-1}_{A(0)} A^{1/2}(0) \nabla u(x).
\end{equation}
Here we are following the notation for the Hessian estimates in~\cite{ErnetaInterior}.

First we prove the bound~\eqref{eq:szz} for $\anew$ in Theorem~\ref{thm:sz}.
This is an analogue of the Sternberg-Zumbrun geometric estimate up to the boundary.
For this, it is convenient to define
the constant vector field
\begin{equation}
\label{def:n:field}
\NN := |e_n|_{A(0)}^{-1}A(0)e_n,
\quad
 \quad  
\NN_i := (a_{nn}(0))^{-1/2} a_{in}(0).
\end{equation}
Notice that $\NN$ has unit norm with respect to the scalar product defined by the inverse matrix $A^{-1}(0)$,
i.e.,
 $|\NN|_{A^{-1}(0)} = 1$.
 Moreover, since $u$ is nonnegative and 
$u = 0$ on $\partial^{0} B_1^{+}$, we have the identity
\begin{equation}
\label{bdy:dir}
|\nabla u|_{A(0)} = \NN \cdot \nabla u \quad \text{ on } \partial^0 B_1^{+}.
\end{equation}
The vector field $\NN$ will also be useful later when controlling the Dirichlet energy on the boundary.

%

\begin{proof}[Proof of \eqref{eq:szz} in Theorem~\ref{thm:sz}]
We test the stability inequality \eqref{stab:half:jac} with a variant of
\[
\cc := | \nabla u|_{A(0)} - \NN \cdot \nabla u,
\]
where $\NN = (\NN_i)$ is the constant vector field defined in \eqref{def:n:field} above.
Since $|\nabla u|_{A(0)}$ is not necessarily smooth when $\nabla u = 0$, 
following~\cite{CabreFigalliRosSerra},
we take a convex $C^{1,1}$ regularization of the modulus $| \cdot |_{A(0)}$ instead.
For each small $\delta > 0$,
we define
\begin{equation}
\label{def:convex:mod}
\phi_{\delta}(z) := |z|_{A(0)} \indicator_{\{|z|_{A(0)} > \delta\}} + \left(\frac{\delta}{2} + \frac{|z|_{A(0)}^2}{2\delta}\right) \indicator_{\{ |z|_{A(0)} < \delta \}}.
\end{equation}

Given that $u$ is nonnegative and 
superharmonic (in the sense that $L u = - f(u) \leq 0$),
unless $u \equiv 0$ (in which case there is nothing to prove),
by the Hopf lemma and uniform ellipticity we have
 $|\nabla u|_{A(0)} \geq c > 0$ on 
 $\partial^0 B_1^{+} \cap \supp \eta$, for some constant $c$.
Hence, for $\delta > 0$ sufficiently small we have
\begin{equation}
\label{choice:delta}
\phi_{\delta}(\nabla u) = |\nabla u|_{A(0)} \quad \text{ in a neighborhood of } \partial^0 B_{1}^{+} \cap \supp \eta \text{ inside } \overline{B_{1}^{+}}.
\end{equation}
Choosing $\delta > 0$ small enough such that \eqref{choice:delta} holds, we let
\[
\cc_{\delta} := \phi_{\delta}(\nabla u) - \NN \cdot \nabla u.
\]
Since $\cc_{\delta}$ vanishes on $\partial^0 B_{1}^{+}$, 
this is a valid test function in the stability inequality \eqref{stab:half:jac}.

We can write the Jacobi operator acting on $\cc_{\delta}$ as the sum of three terms
\[
\begin{split}
\cc_{\delta} \jacobi \cc_{\delta} 
&= \cc_{\delta}(L \cc_{\delta} + f'(u)\cc_{\delta})\\
&= \phi_{\delta}(\nabla u)\jacobi \left[\phi_{\delta}(\nabla u) \right]
- \cc_{\delta} \jacobi \left[\NN \cdot \nabla u\right] 
-(\NN \cdot \nabla u) \jacobi\left[ \phi_{\delta}(\nabla u)\right].
\end{split}
\]
Multiplying this identity by $\eta^2$ and integrating in $B_1^{+}$ yields the left-hand side of \eqref{stab:half:jac},~i.e.,
\begin{equation}\label{bdy:curv:1}
\begin{split}
\int_{B_1^{+}}  \cc_{\delta} \jacobi \, \cc_{\delta} \eta^2\d x 
&=  \int_{B_1^{+}}  \phi_{\delta}(\nabla u)\jacobi \left[\phi_{\delta}(\nabla u) \right] \eta^2\d x  - \int_{B_1^{+}}  \cc_{\delta} \jacobi \left[\NN \cdot \nabla u\right] \eta^2 \d x \\
& \quad \quad \quad  - \int_{B_1^{+}}  (\NN \cdot \nabla u) \jacobi\left[ \phi_{\delta}(\nabla u)\right] \eta^2 \d x .
\end{split}
\end{equation}
We now study each of the three terms in \eqref{bdy:curv:1} separately.

\vspace{3mm}\noindent
\textbf{First term}. \textit{ We prove that}
\begin{equation}
\label{term1}
\begin{split}
&\int_{B_1^{+}}\phi_{\delta}(\nabla u)  \jacobi \left[\phi_{\delta}(\nabla u) \right]  \eta^2 \d x\\
& \quad \quad \geq \int_{B_1^{+}}\anew^2 \,  \indicator_{\{|\nabla u|_{A(0)} > \delta\}} \, \eta^2\d x
- C \delta \int_{B_{1}^{+}} |f'(u)| ( |\nabla u| + \delta) \eta^2 \d x\\
& \quad \quad \quad \quad  -C \varepsilon \int_{B_1^{+}} \left( |\nabla u| + \delta \right) \left( |D^2 u |  \eta^2 + |\nabla u| |\nabla (\eta^2)| \right) \d x 
- C \varepsilon \int_{\partial^{0} B_1^{+}}|\nabla u|^2 \eta^2 \d \hcal^{n-1}.
\end{split}
\end{equation}

Recalling that $\jacobi = L + f'(u)$, we start by computing $L[\phi_{\delta}(\nabla u)]$ first.
Here, since $u \in W^{3, p}(B_1^+)$ with $p > n$, by Sobolev embedding $u \in C^2(\overline{B_1^+})$ and, moreover, the Hessian $D^2 u$ is differentiable a.e. in $B_1^+$ (for instance, see~\cite{EvansGariepy}).
Thus we have
\begin{equation}
\label{uta4}
\begin{split}
L \left[ \phi_{\delta}(\nabla u) \right] 
&= a_{ij}(x) \partial_{ij}^2\left[ \phi_{\delta}(\nabla u) \right] + \vv_{i}(x) \partial_{i}\left[ \phi_{\delta}(\nabla u) \right]\\
&= a_{ij}(x) \partial_{z_k}\phi_{\delta}(\nabla u) u_{ijk} 
+a_{ij}(x) \partial^2_{z_k z_l}\phi_{\delta}(\nabla u) u_{jk} u_{il}+ \vv_{i}(x) \partial_{z_k} \phi_{\delta}(\nabla u) u_{ik}
\end{split}
\end{equation}
a.e.~in~$B_1^{+}$.
By the convexity of $\phi_{\delta}$ and its definition~\eqref{def:convex:mod},
it is easy to check that
\begin{equation}
\label{uta5}
\phi_{\delta}(\nabla u) a_{ij}(x) \partial^2_{z_k z_l}\phi_{\delta}(\nabla u) u_{jk} u_{il} \geq \anew^2 \,  \indicator_{\{|\nabla u|_{A(0)} > \delta\}}.
\end{equation}
Therefore, multiplying \eqref{uta4} by $\phi_{\delta}(\nabla u) \eta^2$, using \eqref{uta5}, and integrating we obtain
\begin{equation}
\label{uta6}
\begin{split}
& \int_{B_1^{+}}\phi_{\delta}(\nabla u)  L \left[\phi_{\delta}(\nabla u) \right]  \eta^2 \d x \\
& \geq \int_{B_1^{+}}a_{ij}(x) \nabla u_{ij} \cdot \nabla \phi_{\delta}(\nabla u) \phi_{\delta}(\nabla u) \eta^2 \d x
+ \int_{B_1^{+}}\anew^2 \,  \indicator_{\{|\nabla u|_{A(0)} > \delta\}} \, \eta^2\d x\\
& \quad \quad \quad + \int_{B_1^{+}} \vv_{i}(x) \partial_{z_k} \phi_{\delta}(\nabla u) u_{ik} \phi_{\delta}(\nabla u) \eta^2 \d x.
\end{split}
\end{equation}

Next, we treat the zero order term $f'(u) \phi_{\delta}(\nabla u)$ in the linearization $\jacobi [\phi_{\delta}(\nabla u)]$.
By direct computation $|\phi_{\delta}(\nabla u) - \nabla \phi_{\delta}(\nabla u) \cdot \nabla u| \leq \delta$ and hence
\begin{equation}
\label{uta1}
\int_{B_1^+} f'(u) \phi_{\delta}(\nabla u)^2 \eta^2 \d x \geq
\int_{B_{1}^{+}} f'(u) \nabla u \cdot \nabla \phi_{\delta}(\nabla u) \phi_{\delta}(\nabla u) \eta^2 \d x
- \delta \int_{B_{1}^{+}} |f'(u)| \phi_{\delta}(\nabla u) \eta^2 \d x.
\end{equation}
Using the equation, we integrate by parts the first term in the right-hand side of~\eqref{uta1} as
\begin{equation}
\label{uta2}
\begin{split}
& \int_{B_{1}^{+}} f'(u) \nabla u \cdot \nabla \phi_{\delta}(\nabla u) \, \phi_{\delta}(\nabla u)\eta^2 \d x 
= \int_{B_{1}^{+}} \nabla [f(u)] \cdot \nabla \phi_{\delta}(\nabla u)  \, \phi_{\delta}(\nabla u)\eta^2 \d x\\
& = \int_{B_{1}^{+}} L u\,  \div\left( \nabla \phi_{\delta}(\nabla u) \, \phi_{\delta}(\nabla u)\eta^2 \right) \d x
- \int_{\partial^0 B_{1}^{+}} f(u) \partial_{z_n} \phi_{\delta}(\nabla u)  \, \phi_{\delta}(\nabla u)\eta^2 \d \hcal^{n-1}.
\\
\end{split}
\end{equation}
Moreover, undoing the integration by parts in
\[
\begin{split}
&  \int_{B_{1}^{+}} a_{ij}(x) u_{ij} \,  \div\left( \nabla \phi_{\delta}(\nabla u) \, \phi_{\delta}(\nabla u)\eta^2 \right) \d x \\
& \quad \quad = - \int_{B_{1}^{+}}  \nabla [a_{ij}(x) u_{ij}]  \cdot \nabla \phi_{\delta}(\nabla u) \, \phi_{\delta}(\nabla u)\eta^2 \d x \\
& \quad \quad \quad \quad \quad +\int_{\partial^0 B_{1}^{+}} a_{ij}(x) u_{ij} \partial_{z_n} \phi_{\delta}(\nabla u)  \, \phi_{\delta}(\nabla u)\eta^2 \d \hcal^{n-1},
 \end{split}
\]
substituting in \eqref{uta2} and using that 
$-L u = f(u)$ up to $\overline{B_1^+}$ by continuity, we deduce
\begin{equation}
\label{uta3}
\begin{split}
& \int_{B_{1}^{+}} f'(u) \nabla u \cdot \nabla \phi_{\delta}(\nabla u) \, \phi_{\delta}(\nabla u)\eta^2 \d x \\
& = - \int_{B_{1}^{+}}  \nabla [a_{ij}(x) u_{ij}]  \cdot \nabla \phi_{\delta}(\nabla u) \, \phi_{\delta}(\nabla u)\eta^2 \d x  
+\int_{B_{1}^{+}} \vv_i(x)  u_i \,  \div\left( \nabla \phi_{\delta}(\nabla u) \, \phi_{\delta}(\nabla u)\eta^2 \right) \d x\\
& \quad \quad \quad \quad +\int_{\partial^0 B_{1}^{+}} \vv_i(x) u_{i} \partial_{z_n} \phi_{\delta}(\nabla u)  \, \phi_{\delta}(\nabla u)\eta^2 \d \hcal^{n-1}.
\end{split}
\end{equation}
Finally, combining \eqref{uta6}, \eqref{uta1}, and \eqref{uta3}, we obtain
\begin{equation}
\label{uta9}
\begin{split}
&\int_{B_1^{+}}\phi_{\delta}(\nabla u)  \jacobi \left[\phi_{\delta}(\nabla u) \right]  \eta^2 \d x
= \int_{B_1^{+}}\phi_{\delta}(\nabla u)  L \left[\phi_{\delta}(\nabla u) \right]  \eta^2 \d x 
+ \int_{B_1^{+}} f'(u) \phi_{\delta}(\nabla u)^2  \eta^2 \d x\\
& \quad \quad \geq \int_{B_1^{+}}\anew^2 \,  \indicator_{\{|\nabla u|_{A(0)} > \delta\}} \, \eta^2\d x 
- \delta \int_{B_{1}^{+}} |f'(u)| \phi_{\delta}(\nabla u) \eta^2 \d x\\
& \quad \quad \quad \quad 
- \int_{B_{1}^{+}}  \nabla a_{ij}(x) u_{ij} \cdot \nabla \phi_{\delta}(\nabla u) \, \phi_{\delta}(\nabla u)\eta^2 \d x  \\
& \quad \quad \quad \quad \quad +\int_{B_{1}^{+}} \vv_i(x) \left\{  u_i \,  \div\left( \nabla \phi_{\delta}(\nabla u) \, \phi_{\delta}(\nabla u)\eta^2 \right) + \partial_{z_k} \phi_{\delta}(\nabla u) u_{ik} \phi_{\delta}(\nabla u)\eta^2 \right\} \d x\\
& \quad \quad \quad \quad \quad \quad
+\int_{\partial^0 B_{1}^{+}} \vv_i(x) u_{i} |e_n|_{A(0)} |\nabla u|_{A(0)}\eta^2 \d \hcal^{n-1},
\end{split}
\end{equation}
where in the boundary term we have used \eqref{choice:delta} and \eqref{bdy:dir}
to write
\[
\partial_{z_n} \phi_{\delta}(\nabla u) \phi_{\delta}(\nabla u) = e_n \cdot A(0) \nabla u = |e_n|_{A(0)} |\nabla u|_{A(0)} \quad \text{ on } \partial^0 B_1^+.
\]
The claim now follows from \eqref{uta9} by applying 
the uniform ellipticity, 
the coefficient bounds 
$\|\nabla a_{ij}\|_{L^{\infty}} + \|\vv_i\|_{C^{0}} \leq \varepsilon$, and
\begin{equation}
\label{uta7}
\phi_{\delta}(\nabla u) \leq C \left( |\nabla u| + \delta\right),
\end{equation}
\begin{equation}
\label{uta8}
|\nabla \phi_{\delta}(\nabla u)| + \phi_{\delta}(\nabla u) |D^2 \phi_{\delta}(\nabla u) | \leq C,
\end{equation} 
where $C$ are universal constants (and hence independent of $\delta$).

\vspace{3mm}\noindent
\textbf{Second term}. \textit{ We prove that}
\begin{equation}
\label{term2}
\begin{split}
&\left|\int_{B_1^{+}}  \cc_{\delta} \jacobi \left[\NN \cdot \nabla u\right] \eta^2 \d x \right|\\
&\quad \quad \quad \quad \leq C \varepsilon \int_{B_1^{+}} \left( |\nabla u| + \delta \right) \left( |D^2 u |  \eta^2  + |\nabla u| |\nabla(\eta^2)|\right) \d x.
\end{split}
\end{equation}

Computing, we have
\begin{equation}
\label{duta1}
L [\NN \cdot \nabla u] = \NN_k a_{ij}(x) u_{ijk} + \NN_k \vv_i(x) u_{ik} \quad \text{ a.e. in } B_1^+.
\end{equation}

Since $\cc_{\delta} = 0$ on $\partial^0 B_1^+$, integrating by parts and using the equation
\begin{equation}
\label{duta2}
\begin{split}
&\int_{B_1^+}   f'(u) \left(\NN \cdot \nabla u\right) \, \cc_{\delta} \eta^2 \d x
= \int_{B_1^+} \NN \cdot \nabla [f(u)]  \, \cc_{\delta}\eta^2 \d x = \int_{B_1^+}  (L u)  \, \NN \cdot \nabla \left( \cc_{\delta}\eta^2 \right) \d x\\
& \quad  \quad  \quad = - \int_{B_1^+} \NN \cdot \nabla[ a_{ij}(x) u_{ij} ]  \,  \cc_{\delta}\eta^2 \d x
+ \int_{B_1^+}  \vv_i(x) u_i  \, \NN \cdot \nabla \left( \cc_{\delta}\eta^2 \right) \d x,
\end{split}
\end{equation}
where in the last line we have integrated by parts again.
Combining \eqref{duta1} and \eqref{duta2}
\begin{equation}
\label{duta3}
\begin{split}
&\int_{B_1^{+}}  \cc_{\delta} \jacobi \left[\NN \cdot \nabla u\right] \eta^2 \d x \\
& = - \int_{B_1^+} \NN \cdot \nabla a_{ij}(x) u_{ij}  \,  \cc_{\delta}\eta^2 \d x 
+ \int_{B_1^+} \NN_k \vv_i(x) u_{ik}\, \cc_{\delta} \eta^2
+ \int_{B_1^+}  \vv_i(x) u_i  \, \NN \cdot \nabla \left( \cc_{\delta}\eta^2 \right) \d x.
\end{split}
\end{equation}
The claim follows from \eqref{duta3} by applying \eqref{uta7}, \eqref{uta8}, and the coefficient bounds.

\vspace{3mm}\noindent
\textbf{Third term}. \textit{ We prove that}
\begin{equation}
\label{term3}
\begin{split}
&\left|\int_{B_1^{+}} (\NN \cdot \nabla u) \jacobi \left[\phi_{\delta}(\nabla u) \right] \eta^2 \d x \right| \\
&  \leq C\int_{B_1^{+}}  (|\nabla u| + \delta) \big( |D^2 u | \{|\nabla(\eta^2)| + \varepsilon \eta^2\} + \varepsilon |\nabla u| |\nabla(\eta^2)| + |\nabla u| |D^2 (\eta^2)|\big) \d x \\
& \quad \quad \quad \quad+ C \int_{\partial^0 B_1^{+}}|\nabla u|^2  \big(|\nabla (\eta^2)|  +\varepsilon\eta^2\big) \d \hcal^{n-1}.
\end{split}
\end{equation}

By definition, we have
\begin{equation}
\label{duta4}
(\NN \cdot \nabla u) \jacobi \left[\phi_{\delta}(\nabla u) \right]  = (\NN \cdot \nabla u) L [\phi_{\delta}(\nabla u)] + \phi_{\delta}(\nabla u) f'(u) (\NN \cdot \nabla u).
\end{equation}
The idea is to integrate 
the first term in \eqref{duta4},
$\int_{B_1^+}  (\NN \cdot \nabla u) L [\phi_{\delta}(\nabla u)] \, \eta^2 \d x$, by parts to get the linearized equation acting on the directional derivative $\NN \cdot \nabla u$ instead of on the modulus $\phi_{\delta}(\nabla u)$.
It will then be easy to bound the remaining terms as in Step 2 above.

We write the operator in divergence form $L u =  \div(A(x) \nabla u) + \bb(x) \cdot \nabla u$ as in \eqref{op:nondiv}.
Integrating by parts twice in $\int_{B_1^{+}} (\NN \cdot \nabla u) \div \big(A(x) \nabla \left[\phi_{\delta}(\nabla u) \right] \big) \eta^2 \d x$,
we have
\begin{equation}
\label{duta5}
\begin{split}
&\int_{B_1^{+}}  (\NN \cdot \nabla u) \div \Big(A(x) \nabla \left[\phi_{\delta}(\nabla u) \right] \Big) \eta^2 \d x\\
&\quad = 
\int_{B_1^{+}} \phi_{\delta}(\nabla u) \div\Big( A(x) \nabla (\NN \cdot \nabla u)\Big) \eta^2 \d x \\
&\quad\quad +\int_{B_1^{+}}  \phi_{\delta}(\nabla u) \Big(2  A(x) \nabla (\NN \cdot \nabla u) \cdot \nabla(\eta^2)  
+ (\NN \cdot \nabla u) \, \div\big\{ A(x) \nabla (\eta^2)\big\}\Big) \d x \\
&\quad\quad +\int_{\partial^0 B_1^{+}} \Big( 
\phi_{\delta}(\nabla u) A(x) \nabla \big\{ (\NN \cdot \nabla u) \eta^2 \big\} \cdot e_n
-(\NN \cdot \nabla u) A(x) \nabla \left[\phi_{\delta}(\nabla u) \right] \cdot e_n \, \eta^2 
\Big)\d \hcal^{n-1}.
\end{split}
\end{equation}
Since $u$ is nonnegative and $u = 0$ on $\partial^0 B_1^+$, we have $\nabla u = |\nabla u|_{A(0)} |e_n|_{A(0)}^{-1} e_n$
and hence, using \eqref{bdy:dir} and \eqref{choice:delta}, the boundary integrand in \eqref{duta5} can be written simply as
\begin{equation}
\label{duta6}
 \phi_{\delta}(\nabla u) A(x) \nabla \big\{ (\NN \cdot \nabla u) \eta^2 \big\} \cdot e_n -(\NN\cdot \nabla u) A(x) \nabla \left[\phi_{\delta}(\nabla u)\right] \cdot e_n \, \eta^2  = |\nabla u|_{A(0)}^2 A(x) \nabla (\eta^2) \cdot e_n.
\end{equation}
Combining \eqref{duta5} and \eqref{duta6}, we deduce
\begin{equation}
\label{duta7}
\begin{split}
& \int_{B_1^+ } (\NN \cdot \nabla u) L [\phi_{\delta}(\nabla u)] \,  \eta^2 \d x\\
& = \int_{B_1^{+}} (\NN\cdot \nabla u) \div \Big(A(x) \nabla \left[\phi_{\delta}(\nabla u)\right] \Big) \, \eta^2 \d x 
+\int_{B_1^{+}} (\NN \cdot \nabla u) \big(\bb(x) \cdot \nabla \left[\phi_{\delta}(\nabla u)\right] \big) \, \eta^2 \d x\\
& = \int_{B_1^{+}} \phi_{\delta}(\nabla u) L [\NN \cdot \nabla u] \, \eta^2 \d x \\
&\quad\quad +\int_{B_1^{+}}  \phi_{\delta}(\nabla u) \Big(2  A(x) \nabla (\NN \cdot \nabla u) \cdot \nabla(\eta^2) + (\NN \cdot \nabla u) \, \div\big\{ A(x) \nabla (\eta^2)\big\}\Big) \d x \\
& \quad \quad \quad +\int_{B_1^{+}} \Big( -\phi_{\delta}(\nabla u) \, \bb(x) \cdot \nabla (\NN \cdot \nabla u) +(\NN \cdot \nabla u) \, \bb(x) \cdot \nabla \left[\phi_{\delta}(\nabla u)\right] \Big) \eta^2 \d x \\
&\quad\quad \quad \quad +\int_{\partial^0 B_1^{+}} |\nabla u|_{A(0)}^2 A(x) \nabla (\eta^2) \cdot e_n \d \hcal^{n-1}.
\end{split}
\end{equation}

We now treat the second term in~\eqref{duta4}.
Integrating by parts twice as in the proof of Step 2 
(this time including boundary terms) and using the equation, it follows that
\begin{equation}
\label{duta8}
\begin{split}
&\int_{B_1^+}   f'(u) \left(\NN \cdot \nabla u\right) \, 
\phi_{\delta}(\nabla u) \eta^2 \d x
= \int_{B_1^+} \NN \cdot \nabla [f(u)]  \, \phi_{\delta}(\nabla u)\eta^2 \d x \\
& \quad = \int_{B_1^+}  (L u)  \, \NN \cdot \nabla \left( \phi_{\delta}(\nabla u)\eta^2 \right) \d x - \int_{\partial^0 B_1^+} f(u) |e_n|_{A(0)} \, \phi_{\delta}(\nabla u) \eta^2
\d \hcal^{n-1}\\
& \quad = - \int_{B_1^+} \NN \cdot \nabla[ a_{ij}(x) u_{ij} ]  \,  \phi_{\delta}(\nabla u)\eta^2 \d x
+ \int_{B_1^+}  \vv_i(x) u_i  \, \NN \cdot \nabla \left( \phi_{\delta}(\nabla u)\eta^2 \right) \d x\\
& \quad \quad \quad \quad 
+ \int_{\partial^0 B_1^+} \vv_i (x) u_i |e_n|_{A(0)} \, \phi_{\delta}(\nabla u) \eta^2 \d \hcal^{n-1}.
\end{split}
\end{equation}

Finally, summing \eqref{duta7} and \eqref{duta8}, we obtain
\begin{equation}
\label{duta9}
\begin{split}
&\int_{B_1^+ } (\NN \cdot \nabla u) \jacobi [\phi_{\delta}(\nabla u)] \,  \eta^2 \d x \\ 
& = \int_{B_1^{+}} 
\Big( L [\NN \cdot \nabla u] 
-\NN \cdot \nabla[ a_{ij}(x) u_{ij} ]
\Big)
\phi_{\delta}(\nabla u)\eta^2 \d x
\\
&\quad\quad +\int_{B_1^{+}}  \phi_{\delta}(\nabla u) \Big(2  A(x) \nabla (\NN \cdot \nabla u) \cdot \nabla(\eta^2) + (\NN \cdot \nabla u) \, \div\big\{ A(x) \nabla (\eta^2)\big\}\Big) \d x \\
& \quad \quad \quad
+\int_{B_1^{+}} \bb(x) \cdot \Big( 
(\NN \cdot \nabla u) \, \nabla \left[\phi_{\delta}(\nabla u)\right] 
-\phi_{\delta}(\nabla u) \, \nabla (\NN \cdot \nabla u) 
\Big) \eta^2 \d x \\
& \quad \quad \quad \quad 
+ \int_{B_1^+}  \vv_i(x) u_i  \, \NN \cdot \nabla \left( \phi_{\delta}(\nabla u)\eta^2 \right) \d x\\
&\quad\quad \quad \quad \quad +\int_{\partial^0 B_1^{+}} \Big(|\nabla u|_{A(0)}^2 A(x) \nabla (\eta^2) \cdot e_n
+\vv_i (x) u_i |e_n|_{A(0)} \, \phi_{\delta}(\nabla u) \eta^2 \Big)
 \d \hcal^{n-1}.
\end{split}
\end{equation}
Noticing that $L [\NN \cdot \nabla u] - \NN \cdot \nabla[ a_{ij}(x) u_{ij} ] = - \NN \cdot \nabla a_{ij}(x) u_{ij} + \NN_k \vv_i(x) u_{ik}$,
every term in the right-hand side of \eqref{duta9} can be bounded as claimed in \eqref{term3}.
For this, apply
the uniform ellipticity, the coefficient bounds $\|\nabla a_{ij}\|_{L^{\infty}} + \|\vv_i\|_{C^{0}} + \|\bb_i\|_{L^{\infty}} \leq 2 \varepsilon$, and the estimates \eqref{uta7} and \eqref{uta8}.

\vspace{3mm}\noindent
\textbf{Conclusion}.
Applying the three estimates \eqref{term1}, \eqref{term2}, and \eqref{term3}
in \eqref{bdy:curv:1} yields the lower bound
\begin{equation}
\label{duta10}
\begin{split}
&\int_{B_1^{+}} \cc_{\delta} \jacobi \cc_{\delta} \,\eta^2 \d x \\
& \quad \geq \int_{B_1^{+} \cap \{|\nabla u|_{A(0)} > \delta\}}  \anew^2\, \eta^2 \d x - C \delta \int_{B_{1}^{+}} |f'(u)| ( |\nabla u| + \delta) \eta^2 \d x\\
& \quad  \quad  \quad - C\int_{B_1^{+}}  (|\nabla u| + \delta) \left( |D^2 u | \{|\nabla(\eta^2)| + \varepsilon \eta^2\} + \varepsilon |\nabla u| |\nabla(\eta^2)| + |\nabla u| |D^2 (\eta^2)|\right) \d x \\
& \quad \quad \quad \quad  \quad - C \int_{\partial^0 B_1^{+}} |\nabla u|^2 (|\nabla (\eta^2)|  +\varepsilon\eta^2) \d \hcal^{n-1}.
\end{split}
\end{equation}
By the integral stability inequality \eqref{stab:half:jac} with $\cc = \cc_{\delta}$, we also have the upper bound
\begin{equation}
\label{duta11}
\begin{split}
\int_{B_1^{+}}  \cc_{\delta} \jacobi \cc_{\delta} \, \eta^2 \d x & 
\leq \int_{B_1^+} \phi_{\delta}(\nabla u)^2 |\nabla \eta - \textstyle\frac{1}{2} \eta A^{-1}(x) \bb(x) |_{A(x)}^2  \d x\\
& \leq C \int_{B_1^+} (|\nabla u| + \delta)^2  \left(|\nabla \eta|^2 + \varepsilon^2 \eta^2\right) \d x.
\end{split}
\end{equation}
Hence, combining 
\eqref{duta10} and \eqref{duta11} and taking the limit as $\delta \to 0$, we deduce the claim
\end{proof}

In order to prove the remaining estimates in Theorem~\ref{thm:sz}, we need to control the right-hand side of~\eqref{eq:szz}.
For this, next we prove two basic Hessian estimates for (generalized) superharmonic functions.
We essentially follow the proof of Theorem~1.2 in \cite{ErnetaInterior}, but now including boundary terms.


\begin{lemma}
\label{lemma:hess:grad}
Let $u \in C^2(\overline{B_1^{+}})$ be superharmonic in the sense that $L u \leq 0$ in $B_1^{+}$,
where $L$ satisfies conditions~\eqref{elliptic} and~\eqref{reg:coeffs} in $\Omega = B_1^{+}$.
Assume that
\[
\|D A\|_{L^{\infty}(B_{1}^{+})} + \|\vv\|_{L^{\infty}(B_{1}^{+})} \leq \varepsilon
\]
for some $\varepsilon > 0$.


Then,
there exists a universal $\varepsilon_0 > 0$ with the following property:
if $\varepsilon \leq \varepsilon_0$,
then, for all $\zeta \in C^{0,1}_c(B_1)$ with $\zeta \geq 0$, we have
\begin{equation}
\label{hess:test}
\begin{split}
\int_{B_1^+}  |D^2 u| \, \zeta \d x &\leq  C \int_{B_1^+} |\nabla u| \, \left(  |\nabla \zeta| + \varepsilon \zeta \right) \d x 
+C \int_{B_1^+} \anew \, \zeta \d x
+ C \int_{\partial^0 B_1^{+}} |\nabla u| \, \zeta \d x\\
\end{split}
\end{equation}
and
\begin{equation}
\label{hessgrad:test}
\begin{split}
\int_{B_1^+}  |D^2 u| |\nabla u| \, \zeta \d x &\leq  C\int_{B_1^+}|\nabla u|^2 \, \left(|\nabla \zeta| + \varepsilon \zeta\right)\d x 
+C \int_{B_1^+} \anew |\nabla u|\, \zeta \d x \\
& \quad \quad \quad \quad \quad + C \int_{\partial^0 B_1^{+}} |\nabla u|^2\zeta \d \hcal^{n-1},
\end{split}
\end{equation}
where $C$ is a universal constant.
\end{lemma}
\begin{proof}
Consider the auxiliary function
\[
\azero := 
\left\{\begin{array}{ll}
\Big( \|A^{1/2}(0)D^2 u A^{1/2}(0)\|_{\rm HS}^2 
- 
|A^{1/2}(0) D^2 u A^{1/2}(0) \nn(x)|^2
\Big)^{1/2} & \text{if } \nabla u \neq 0\\
        0 & \text{if } \nabla u = 0,
\end{array}\right.
\]
where the vector field 
$\nn(x)$ has been introduced in \eqref{n:vector} in the definition of $\anew$ in \eqref{def:anew}.
Using that $\|D A\|_{L^{\infty}(B_1^{+})} \leq \varepsilon$,
it is easy to show (see \cite{ErnetaInterior}) that
\begin{equation}
\label{tim1}
|\anew^2 - \azero^2| \leq C \varepsilon |x| \azero^2 \quad \text{ in } B^{+}_1,
\end{equation}
where $C$ always denotes a universal constant.
In particular, the functions $\anew$ and $\azero$ are comparable for $\varepsilon$ small.
Using that $L u \leq 0$, following \cite{ErnetaInterior}, it is not hard to show that 
\begin{equation}
\label{tim2}
|D^2 u| \leq - C \tr\big( A(0) D^2 u\big) + C \azero + C \varepsilon |x| |D^2 u| + C \varepsilon |\nabla u| \quad \text{ a.e. in } B_1^{+}.
\end{equation}

First we prove the Hessian bound~\eqref{hess:test}.
Multiplying \eqref{tim2} by $\zeta$ and integrating in $B_1^{+}$
\begin{equation}
\label{tim3}
\begin{split}
\int_{B_1^+}  |D^2 u| \zeta \d x &\leq - C \int_{B_1^+} \tr\big( A(0) D^2 u\big) \, \zeta \d x +C \int_{B_1^+} \azero \, \zeta \d x\\
& \quad \quad \quad 
+C \varepsilon \int_{B_1^+} |x| |D^2 u| \zeta \d x+C \varepsilon \int_{B_1^+}  |\nabla u|  \, \zeta \d x.
\end{split}
\end{equation}
Integrating by parts, we have
\[
-\int_{B_1^+} \tr\big( A(0) D^2 u\big) \, \zeta \d x = \int_{B_1^{+}} A(0) \nabla u \cdot \nabla \zeta \d x - \int_{\partial^0 B_1^{+}} A(0) \nabla u \cdot e_n \, \zeta \d \hcal^{n-1},
\]
and substituting in~\eqref{tim3}, by uniform ellipticity,
\begin{equation}
\label{tim4}
\begin{split}
\int_{B_1^+}  |D^2 u| \zeta \d x &\leq  C \int_{B_1^+} |\nabla u| \, |\nabla \zeta| \d x 
+C \int_{B_1^+} \azero \, \zeta \d x
+ C \int_{\partial^0 B_1^{+}} |\nabla u| \, \zeta \d x
\\
&\quad \quad \quad  + C \varepsilon \int_{B_1^+} |x| |D^2 u| \zeta \d x+C \varepsilon \int_{B_1^+}  |\nabla u|  \, \zeta \d x.
\end{split}
\end{equation}
Choosing $\varepsilon_0 > 0$ universal sufficiently small,
we can absorb the Hessian term in the right-hand side of \eqref{tim4},
and by \eqref{tim1} (taking $\varepsilon_0$ smaller)
we deduce the first claim.

For the second estimate~\eqref{hessgrad:test},
multiplying \eqref{tim2} by $|\nabla u|_{A(0)} \zeta$ and integrating in $B_1^{+}$
\begin{equation}
\label{tim5}
\begin{split}
\int_{B_1^+}  |D^2 u| |\nabla u|_{A(0)} \zeta \d x &\leq - C \int_{B_1^+} |\nabla u|_{A(0)} \tr\big( A(0) D^2 u\big) \, \zeta \d x 
+C \int_{B_1^+} \azero |\nabla u|_{A(0)}\, \zeta \d x\\
&\quad \quad \quad  + C \varepsilon \int_{B_1^+} |x| |D^2 u| |\nabla u|_{A(0)} \zeta \d x+C \varepsilon \int_{B_1^+}  |\nabla u|^2  \, \zeta\d x.
\end{split}
\end{equation}
The first integrand in the right-hand side of \eqref{tim5} can be bounded by
\begin{equation}
\label{tim6}
\begin{split}
-|\nabla u|_{A(0)} \tr\big(A(0) D^2 u\big) &\leq  - \frac{1}{2}\div \big(|\nabla u|_{A(0)} A(0)\nabla u \big) +C  \azero |\nabla u|_{A(0)} \\
\end{split}
\quad \text{ a.e. in } B_1^+.
\end{equation}
Substituting \eqref{tim6} in \eqref{tim5} leads to
\[
\begin{split}
\int_{B_1^+}  |D^2 u| |\nabla u|_{A(0)} \zeta \d x &\leq 
- C\int_{B_1^+}\div \big(|\nabla u|_{A(0)} A(0)\nabla u \big) \,\zeta \d x 
+C \int_{B_1^+} \azero |\nabla u|_{A(0)}\, \zeta \d x\\
&\quad \quad \quad + C \varepsilon \int_{B_1^+} |x| |D^2 u| |\nabla u|_{A(0)}  \, \zeta\d x +C \varepsilon \int_{B_1^+}  |\nabla u|^2  \, \zeta\d x,
\end{split}
\]
and integrating by parts the divergence term, we obtain the inequality
\begin{equation}
\label{tim7}
\begin{split}
\int_{B_1^+}  |D^2 u| |\nabla u|_{A(0)} \zeta \d x &\leq  C\int_{B_1^+}|\nabla u|^2 \, \left(|\nabla \zeta| + \varepsilon \zeta\right)\d x 
+C \int_{B_1^+} \azero |\nabla u|_{A(0)}\, \zeta \d x\\
&\quad \quad \quad  + C \varepsilon \int_{B_1^+} |x| |D^2 u| |\nabla u|_{A(0)}  \, \zeta\d x + C \int_{\partial^0 B_1^{+}} |\nabla u|^2\zeta \d \hcal^{n-1}.
\end{split}
\end{equation}
Once again, choosing $\varepsilon_0 > 0$ universal small, we can absorb the ``Hessian times the gradient'' error in \eqref{tim7} into the left-hand side,
and by \eqref{tim1} we deduce the second claim.
\end{proof}

Thanks to Lemma~\ref{lemma:hess:grad}, we can get rid of the Hessian terms appearing
in the right-hand side of the first inequality~\eqref{eq:szz} in Theorem~\ref{thm:sz}:

\begin{lemma}
\label{lemma:sz2}
Let $u \in W^{3,p}(B_1^+)$, for some $p > n$,
be a nonnegative stable solution of 
$- L u = f(u)$ in $B_{1}^{+}$, with $u = 0$ on $\partial^0 B_{1}^{+}$.
Assume that $f \in C^1(\R)$ is nonnegative.
Assume that $L$ satisfies conditions~\eqref{elliptic},~\eqref{reg:coeffs}, and~\eqref{reg:b} in $\Omega = B_1^{+}$, and that
\[
\|D A\|_{L^{\infty}(B_{1}^{+})} + \|\vv\|_{L^{\infty}(B_{1}^{+})} \leq \varepsilon
\]
for some $\varepsilon > 0$.

If $\varepsilon \leq \varepsilon_0$, then
\[
\int_{B_{8/9}^{+}} \anew^2  \d x \leq C \int_{B_1^{+}} |\nabla u|^2 \d x + C \int_{\partial^0 B_1^{+}} |\nabla u|^2 \d \hcal^{n-1},
\]
where $\varepsilon_0 > 0$ and $C$ are universal constants.
\end{lemma}
\begin{proof}
Let $\varepsilon_0 > 0$ be the universal constant in the conclusion of Lemma~\ref{lemma:hess:grad}.
Applying~\eqref{hessgrad:test} in Lemma~\ref{lemma:hess:grad} with 
$\zeta = |\nabla (\eta^2)| + \varepsilon \eta^2 \in C^{0,1}_c(B_1)$
yields
\begin{equation}
\label{ze1}
\begin{split}
&\int_{B_1^+}  |D^2 u| |\nabla u| (|\nabla (\eta^2)| + \varepsilon \eta^2) \d x \\
&\leq  C\int_{B_1^+}|\nabla u|^2 \, \left(
|D^2 (\eta^2)| + \varepsilon |\nabla (\eta^2)| + \varepsilon^2 \eta^2\right)\d x 
+C \int_{B_1^+} \anew |\nabla u|\, (|\nabla (\eta^2)| + \varepsilon \eta^2) \d x \\
& \quad \quad \quad \quad \quad + C \int_{\partial^0 B_1^{+}} |\nabla u|^2 (|\nabla (\eta^2)| + \varepsilon \eta^2) \d \hcal^{n-1}.
\end{split}
\end{equation}
Since $|\nabla (\eta^2)| + \varepsilon \eta^2 = |\eta| \left( 2|\nabla \eta| + \varepsilon |\eta| \right)$, by Cauchy-Schwarz, the second term in \eqref{ze1} can be bounded by
\begin{equation}
\label{ze2}
\int_{B_1^+} \anew |\nabla u|\, (|\nabla (\eta^2)| + \varepsilon \eta^2) \d x \leq C \left( \int_{B_1^+} \anew^2 \eta^2\d x\right)^{1/2} \left( \int_{B_1^{+}} |\nabla u|^2\, \big(|\nabla \eta|^2 + \varepsilon^2 \eta^2\big) \d x \right)^{1/2}.
\end{equation}
Hence, applying \eqref{ze1} and \eqref{ze2} 
to the Hessian errors in the right-hand side of \eqref{eq:szz} in Theorem~\ref{thm:sz},
we obtain
\begin{equation}
\label{ze3}
\begin{split}
\int_{B_1^{+}}  \anew^2 \eta^2 \d x 
& \leq  C \left( \int_{B_1^+} \anew^2 \eta^2\d x\right)^{1/2} \left( \int_{B_1^{+}} |\nabla u|^2\, \big(|\nabla \eta|^2 + \varepsilon^2 \eta^2\big) \d x \right)^{1/2}\\
& \quad \quad \quad + C \int_{B_{1}^{+}} |\nabla u|^2 \left( |\nabla \eta|^2 + |D^2 (\eta^2)|+ \varepsilon |\nabla (\eta^2)| + \varepsilon^2 \eta^2\right) \d x\\
& \quad \quad \quad \quad \quad + C \int_{\partial^0 B_1^{+}} |\nabla u|^2 (|\nabla (\eta^2)|  +\varepsilon\eta^2) \d \hcal^{n-1}.
\end{split}
\end{equation}
Therefore, by Young's inequality, we can absorb the $\int_{B_1^{+}} \anew^2\, \eta^2 \d x$ term in \eqref{ze3} into the left-hand side. 
Choosing $\eta \in C^{\infty}_c(B_1)$ with $0 \leq \eta \leq 1$ in $B_1$ and 
$\eta = 1$ in $B_{8/9}$, 
by the universal bound $\varepsilon \leq \varepsilon_0$, we deduce the claim.
\end{proof}

Thanks to the preliminary lemmas above, we are now in position to conclude the proof of Theorem~\ref{thm:sz}:

\begin{proof}[Proof of the boundary estimates~\eqref{ineq:pohozaev}, \eqref{ineq:hessgrad}, \eqref{ineq:szkinda}, and \eqref{ineq:hessian} in
Theorem~\ref{thm:sz}]
Once we obtain the boundary gradient estimate~\ref{ineq:pohozaev},
the remaining inequalities \eqref{ineq:hessgrad}, \eqref{ineq:szkinda}, and \eqref{ineq:hessian}
will follow easily from Lemmas~\ref{lemma:hess:grad} and~\ref{lemma:sz2}.

To control the gradient on the boundary, we proceed in two steps.
First we employ the Pohozaev trick to bound the $L^2$ norm of $\nabla u$ on the lower boundary by the Dirichlet energy up to Hessian errors.
Secondly, we use Lemmas~\ref{lemma:hess:grad} and~\ref{lemma:sz2} to control these Hessian errors and apply Simon's lemma~(recalled in Appendix~\ref{app:simon}).

\vspace{3mm}\noindent
\textbf{Step 1.}
\textit{We prove that }
\[
\|\nabla u\|_{L^2(\partial^0 B_{2/3}^{+})}^2 \leq C (1 + \varepsilon) \|\nabla u\|^2_{L^2(B_{7/9}^{+})} + C \varepsilon \| |D^2 u| \, |\nabla u|\|_{L^1(B_{7/9}^{+})},
\]
\textit{where $C$ is a universal constant.}
 
Let $\eta \in C^{\infty}_c(B_{7/9})$.
Integrating by parts, by the properties of $u$ and the vector field $\NN$ defined in \eqref{def:n:field}, it is easy to check that
\begin{equation} \label{pohozaev:1}
\begin{split}
&|e_n|_{A(0)} \int_{\partial^0 B_1^+} |\nabla u |^2_{A(0)} \eta^2 \d \hcal^{n-1} \\
&\quad =\int_{B_1^+}  \div\left (| \nabla u|_{A(0)}^2 \NN - 2 (\NN \cdot \nabla u) A(0) \nabla u \right ) \eta^2 \d x \\
& \quad \quad \quad \quad + \int_{B_1^+}   \left (| \nabla u|_{A(0)}^2 \NN 
- 2 (\NN \cdot \nabla u) A(0) \nabla u \right) \cdot \nabla (\eta^2)\d x.
\end{split}
\end{equation}
The divergence term in \eqref{pohozaev:1} can be written as
 \[
\begin{split}
&\div\left (| \nabla u|_{A(0)}^2 \NN
- 2 (\NN \cdot \nabla u) A(0) \nabla u \right )  = 
-2 (\NN \cdot \nabla u) \tr (A(0) D^2 u)\\
&\quad = 
-2 (\NN \cdot \nabla u) L u +2 (\NN \cdot \nabla u) (\vv(x) \cdot \nabla u) +2 (\NN \cdot \nabla u) \tr\big(\{A(x) - A(0)\} D^2 u\big)\\
& \quad \leq - 2 (\NN \cdot \nabla u) L u + C \varepsilon |\nabla u|^2 + C \varepsilon |x|  |D^2 u| |\nabla u|,
\end{split}
\]
where in the last line we have used the bounds 
$\|\vv\|_{L^{\infty}(B_1^+)}\leq \varepsilon$ 
and $|A(x) - A(0)| \leq \varepsilon |x|$ for $x \in B_1^+$.
It follows that 
\begin{equation} \label{pohozaev:2}
\begin{split}
&|e_n|_{A(0)} \int_{\partial^0 B_1^+} |\nabla u |^2_{A(0)} \eta^2 \d \hcal^{n-1}  \\
& \quad \quad \leq 
-2 \int_{B_1^+}  (\NN \cdot \nabla u) Lu \, \eta^2 \d x + C \int_{B_1^{+}} |\nabla u|^2 \big(|\nabla (\eta^2)| + \varepsilon \eta^2\big) \d x \\
& \quad \quad \quad \quad \quad \quad \quad + C \varepsilon \int_{B_1^{+}} |x| |D^2 u| |\nabla u| \eta^2 \d x
\end{split}
\end{equation}
and, thus, it remains to control the term $-2 \int_{B_1^+} (\NN \cdot \nabla u) L u \, \eta^2 \d x$ in \eqref{pohozaev:2}.

Since $-Lu = f(u)$ in $B_1^+$, the primitive $F(t) := \int_{0}^{t}f(s) \d s$ of $f$ satisfies
\[
\NN \cdot \nabla [F(u)] = (\NN \cdot \nabla u) f(u) = -(\NN\cdot \nabla u) Lu,
\]
and the first term on the right hand side of \eqref{pohozaev:2} can be integrated by parts as
\begin{equation}\label{pohozaev:3}
- \int_{B_1^{+}} (\NN \cdot \nabla u) Lu \, \eta^2\d x = \int_{B_1^{+}} \NN \cdot \nabla [F(u)] \eta^2 \d x= -\int_{B_1^{+}} F(u) \big( \NN \cdot \nabla (\eta^2) \big)\d x. 
\end{equation}
By the monotonicity of $f$, since $u$ and $f$ are nonnegative, we have $|F(u) |\leq u f(u) =~- uL u$.
Hence, writing $L$ in divergence form $L u = \div(A(x) \nabla u) + \bb(x) \cdot \nabla u$ as in~\eqref{op:nondiv}, by the coefficient bound 
$\|\bb\|_{L^{\infty}(B_1^+)} \leq C \varepsilon$ 
we deduce
\begin{equation}
\label{bigfbound}
|F(u)| \leq - u \, \div (A(x) \nabla u) + C \varepsilon u \, |\nabla u|.
\end{equation}
Using \eqref{bigfbound}, 
we estimate the right-hand side of \eqref{pohozaev:3} by
\begin{equation}\label{pohozaev:4}
\Big|-\int_{B_1^{+}} F(u) \big( \NN \cdot \nabla (\eta^2) \big) \d x \Big| \leq - C\int_{B_1^{+}} u \, \div\big(A(x)\nabla u\big) |\nabla (\eta^2)| \d x + C \varepsilon \int_{B_1^{+}} u \, |\nabla u| |\nabla (\eta^2)| \d x,
\end{equation}
and since $|\nabla(\eta^2)|$ 
is Lipschitz,
the divergence term in \eqref{pohozaev:4} can be integrated by parts as
\begin{equation}\label{pohozaev:5}
\begin{split}
-\int_{B_1^{+}} u \, \div\big(A(x)\nabla u\big) |\nabla (\eta^2)| \d x 
&= \int_{B_1^{+}} | \nabla u|_{A(x)}^2|\nabla(\eta^2)| \d x + \int_{B_1^{+}} u \,A(x) \nabla u \cdot \nabla |\nabla(\eta^2)| \d x.
\end{split}
\end{equation}
Therefore, combining 
\eqref{pohozaev:3}, \eqref{pohozaev:4}, and \eqref{pohozaev:5}, we deduce
\begin{equation}
\label{poho}
-\int_{B_1^+}   (\NN \cdot \nabla u) Lu \, \eta^2\d x \leq
\int_{B_1^{+}} | \nabla u|_{A(x)}^2|\nabla(\eta^2)| \d x 
+ C \int_{B_1^{+}} u |\nabla u| \left( |D^2(\eta^2)| + \varepsilon |\nabla (\eta^2)| \right)\d x.
\end{equation}
Moreover,
we can bound the last term in \eqref{poho}
by Cauchy-Schwarz and the Poincaré inequality (valid since $u = 0$ on $\partial^0 B_1^+$) as
\begin{equation}
\label{poho:2}
\begin{split}
&\int_{B_1^{+}} u |\nabla u| \left( |D^2(\eta^2)| + \varepsilon |\nabla (\eta^2)| \right)\d x\\
&\leq C\left( \int_{B_{7/9}^{+}} |\nabla u|^2 \d x \right)^{1/2} \left(\int_{B_1^{+}} |\nabla u|^2 \left( |D^2(\eta^2)| + \varepsilon |\nabla (\eta^2)| \right)^2 \d x \right)^{1/2}.
\end{split}
\end{equation}

Applying the bounds \eqref{poho} and \eqref{poho:2} in \eqref{pohozaev:2}, by uniform ellipticity, we obtain
\begin{equation} \label{pohozaev:6}
\begin{split}
&\int_{\partial^0 B_1^+}  |\nabla u |^2  \eta^2\d \hcal^{n-1}  \\
& \quad \quad \leq 
C\left( \int_{B_{7/9}^{+}} |\nabla u|^2 \d x \right)^{1/2} \left(\int_{B_1^{+}} |\nabla u|^2 \left( |D^2(\eta^2)| + \varepsilon |\nabla (\eta^2)| \right)^2 \d x \right)^{1/2}\\
& \quad \quad \quad  \quad \quad  \quad \quad  + C \int_{B_1^{+}} |\nabla u|^2 \big(|\nabla (\eta^2)| + \varepsilon \eta^2\big) \d x + C \varepsilon \int_{B_1^{+}} |x| |D^2 u| |\nabla u| \eta^2 \d x.
\end{split}
\end{equation}
Finally, choosing 
$\eta \in C^{\infty}_c(B_{7/9})$ in \eqref{pohozaev:6} satisfying
$\eta = 1$ in $B_{2/3}$
and $0 \leq \eta \leq 1$ in $B_{7/9}$,
we deduce
\[
\int_{\partial^0 B_{2/3}^+} |\nabla u |^2 \d \hcal^{n-1}  \leq  
C \left(1 + \varepsilon\right) \int_{B_{7/9}^{+}} |\nabla u|^2 \d x+ C \varepsilon \int_{B_{7/9}^{+}} |x| |D^2 u | |\nabla u| \d x,
\]
which yields the claim.

\vspace{3mm}\noindent
\textbf{Step 2.}
\textit{Conclusion.}

Let $\varepsilon_0 > 0$ be the universal constant in the conclusion of Lemma~\ref{lemma:hess:grad}.
Applying this result with a cut-off $\zeta \in C^{1}_c(B_{8/9})$ such that $0 \leq \zeta \leq 1$ and $\zeta = 1$ in $B_{7/9}$,
if $\varepsilon \leq \varepsilon_0$, then
\begin{equation}
\label{arr0}
\| |D^2 u| \, |\nabla u|\|_{L^1(B_{7/9}^{+})} \leq C \|\nabla u\|_{L^2(B_{8/9}^{+})}^2 + C \|\nabla u\|_{L^2(\partial^0 B_{8/9}^{+})}^2 + C \| \anew \, |\nabla u|\|_{L^1(B_{8/9}^{+})}.
\end{equation}
Hence, applying Cauchy-Schwarz in \eqref{arr0} and by Lemma~\ref{lemma:sz2}, we deduce
\begin{equation}
\label{arr1}
\| |D^2 u| \, |\nabla u|\|_{L^1(B_{7/9}^{+})} \leq C \|\nabla u\|_{L^2(B_{1}^{+})}^2 + C \|\nabla u\|_{L^2(\partial^0 B_{1}^{+})}^2.
\end{equation}

Let $\delta > 0$.
Using \eqref{arr1} in Step $1$ above,
letting $\varepsilon_{\delta} := \min \{\varepsilon_0, \delta/ C \}$, we obtain
\begin{equation}
\label{arr2}
\|\nabla u\|_{L^2(\partial^0 B_{2/3}^{+})}^2 \leq \delta \|\nabla u\|^2_{L^2(\partial^0 B_1^{+})} + C \|\nabla u\|_{L^2(B_1^{+})}^2
\quad \text{ for } \varepsilon \leq \varepsilon_{\delta}.
\end{equation}
Hence, by translation and rescaling of \eqref{arr2}, for all $y \in \partial^0 B_1^{+}$ and $\rho > 0$ such that $B^{+}_{\rho}(y) \subset B_1^{+}$, we have
\begin{equation}
\label{arr3}
\begin{split}
\rho \int_{\partial^{0} B_{2\rho/3}^{+}(y)} |\nabla u|^2 \d \hcal^{n-1} &\leq \delta \rho \int_{\partial^0 B^{+}_{\rho}(y)} |\nabla u|^2 \d \hcal^{n-1} + C \int_{B^{+}_\rho(y)} |\nabla u|^2 \d x\\
&\leq \delta \rho \int_{ \partial^0 B^{+}_{\rho}(y)} |\nabla u|^2 \d \hcal^{n-1} + C \int_{B_1^{+}} |\nabla u|^2 \d x \quad \quad 
\text{ for } \varepsilon \leq \varepsilon_{\delta}.
\end{split}
\end{equation}

Since $y \in \partial^0 B_1^{+}$,
we have $y = (y', 0)$ for some $y' \in \R^{n-1}$, and
the lower boundary $\partial^0 B_{\rho}^{+}(y)$ is simply the $(n-1)$-dimensional ball 
$B'_{\rho}(y') := \{x \in \R^{n-1} \colon |x - y'| < \rho\} \subset \R^{n-1} = \partial^0 \R^{n}$.
By \eqref{arr3}, we can apply the Simon lemma to the subadditive quantity 
\[
B' \mapsto \int_{B'} |\nabla u|^2 \d \hcal^{n-1}
\]
on balls $B' \subset B_1' \subset \R^{n-1} = \partial^0 \R^{n}$ to deduce the bound
\begin{equation}
\label{real:poho}
\int_{\partial^0 B_{2/3}^{+}} |\nabla u|^2 \d \hcal^{n-1} \leq C \int_{B_1^{+}} |\nabla u|^2 \d x \quad \text{ for } \varepsilon \leq \varepsilon_{\delta},
\end{equation}
for some universal $\delta >0$.
In particular, we may take $\varepsilon_0$ universal
equal to $\varepsilon_{\delta}$
and this concludes the proof of \eqref{ineq:pohozaev}.


Finally, to deduce the remaining Hessian estimates we proceed as in the proof of \eqref{arr1}.
To prove \eqref{ineq:hessgrad},
we apply \eqref{hessgrad:test} from Lemma~\ref{lemma:hess:grad} with a cut-off function $\zeta \in C^{1}_c(B_{16/27})$ such that $0 \leq \zeta \leq 1$ and $\zeta = 1$ in $B_{4/7 = 16/28} \subset B_{16/27}$, and by Cauchy-Schwarz
\begin{equation}
\label{arrbis}
\begin{split}
\| |D^2 u| \, |\nabla u|\|_{L^1(B_{4/7}^{+})} 
&\leq C \|\nabla u\|_{L^2(B_{16/27}^{+})}^2 + C \|\nabla u\|_{L^2(\partial^0 B_{16/27}^{+})}^2 
+ C \|\anew\|_{L^2(B_{16/27}^{+})}^2\\
&\leq 
C \|\nabla u\|_{L^2(B_{2/3}^{+})}^2
+ C \|\nabla u\|_{L^2(\partial^0 B_{2/3}^{+})}^2,
\end{split}
\end{equation}
where in the last line we have used Lemma~\ref{lemma:sz2} applied to the rescaled function $u(\frac{2}{3} \cdot)$.
Applying \eqref{real:poho} to \eqref{arrbis} now leads to \eqref{ineq:hessgrad}.

Now, the bound~\eqref{ineq:szkinda} is easily obtained combining Lemma~\ref{lemma:sz2} with the boundary estimate~\eqref{ineq:pohozaev}.
The final estimate~\eqref{ineq:hessian} follows from Lemma~\ref{lemma:hess:grad} and the above.
\end{proof}


\section{Boundary $W^{1, 2+\gamma}$ estimate}
\label{section:higher:int}

First we control the Dirichlet energy by the $L^1$ norm of the solution under a smallness condition on the coefficients.
This follows from Theorem~\ref{thm:sz} and the interpolation inequalities of Cabr\'{e} in \cite{CabreRadial} (recalled in Appendix~\ref{app:interpolation} below).

\begin{lemma}
\label{lemma:l2l1}
Let $u \in W^{3,p}(B_1^{+})$,
for some $p > n$,
be a nonnegative stable solution of 
$- L u = f(u)$ in $B_{1}^{+}$, with $u = 0$ on $\partial^0 B_{1}^{+}$.
Assume that $f \in C^1(\R)$ is nonnegative and nondecreasing.
Assume that $L$ satisfies conditions~\eqref{elliptic},~\eqref{reg:coeffs}, and~\eqref{reg:b} in $\Omega = B_1^{+}$, and
\[
\|D A\|_{L^{\infty}(B_{1}^{+})} + \|\vv\|_{L^{\infty}(B_{1}^{+})} \leq \varepsilon
\]
for some $\varepsilon > 0$.

If $\varepsilon \leq \varepsilon_0$, then
\[
 \|\nabla u\|_{L^{2}(B^{+}_{1/2})}  \le C \| u\|_{L^{1}(B^{+}_{1})},
\]
where $\varepsilon_0 > 0$ and $C$ are universal constants.
\end{lemma}
\begin{proof}
We cover $B_{1/2}^{+}$ 
(except for a set of measure zero) 
with a family of disjoint open cubes $Q_j \subset \R^{n}_{+}$
of the same side-length and small enough so that $Q_j\subset B_{4/7}^{+}$.
The side-length and the number of cubes depend only on $n$.
Combining the interpolation inequalities of Proposition \ref{interpol} (with $p = 2$) and Proposition \ref{nash}, rescaled from the unit cube to $Q_j$, 
with 
$\tilde\delta= \delta^{3/2}$ for a given $\delta\in(0,1)$,
we have
\begin{equation*}
\int_{Q_j}|\nabla u|^{2}dx \leq C\delta \int_{Q_j}\lvert D^2u\rvert |\nabla u| \,dx+ C\delta \int_{Q_j}|\nabla u|^2dx+C\delta^{-2-\frac{3n}{2}}\left( \int_{Q_j}|u|\,dx\right)^2.
\end{equation*} 
Since $Q_j\subset B_{4/7}^{+}$,
applying \eqref{ineq:hessgrad} from 
Theorem~\ref{thm:sz},
for $\varepsilon \leq \varepsilon_0$ we deduce
\begin{equation*}
\int_{Q_j}|\nabla u|^{2}dx \leq C \delta \int_{B_1^{+}}|\nabla u|^2dx+C\delta^{-2-\frac{3n}{2}}\left( \int_{B_1^{+}}|u|\,dx\right)^2.
\end{equation*} 
Adding up these inequalities, we obtain
\begin{equation}\label{rofnew}
\|\nabla u\|_{L^{2}(B_{1/2}^{+})}^2 \le C \delta  \|\nabla u\|_{L^{2}(B_{1}^{+})}^2 + C\delta^{-2-\frac{3n}{2}} \|u\|_{L^{1}(B_{1}^{+})}^2 \quad \text{ for } \delta \in (0, 1) \text{ and } \varepsilon \leq \varepsilon_0.
\end{equation}

For $B_{\rho}^{+}(y) \subset B_1^{+}$ with $y \in \partial^0 B_1^{+}$, 
the function
$u^{y,\rho} :=u(y+\rho \, \cdot)$ 
is a stable solution to a semilinear equation with coefficients $A^{y,\rho} = A(y + \rho \, \cdot)$ and $\vv^{y,\rho} = \rho\, \vv(y + \rho \, \cdot)$.
In particular, since $\rho \leq 1$, for $\varepsilon \leq \varepsilon_0$ we have
\[
\| D A^{y,\rho} \|_{L^{\infty}(B_{1}^{+})} + \| \vv^{y,\rho} \|_{L^{\infty}(B_{1}^{+})} \leq \rho \varepsilon \leq \varepsilon_0,
\]
and we may apply \eqref{rofnew} to $u^{y,\rho}$, which yields
\[
\rho^{n+2}\int_{B_{\rho/2}^{+}(y)}|\nabla u|^2\,dx \leq 
C\delta \rho^{n+2}\int_{B_{\rho}^{+}(y)}|\nabla u|^2\,dx
+ C\delta^{-2-\frac{3n}{2}}\left(\int_{B_{\rho}^{+}(y)}|u|\,dx\right)^2,
\]
hence
\begin{equation}
\label{dec1}
\begin{split}
\rho^{n+2}\int_{B_{\rho/2}^{+}(y)}|\nabla u|^2\,dx
\leq & C \delta \rho^{n+2}\int_{B_{\rho}^{+}(y)}|\nabla u|^2\,dx
+C\delta^{-2-\frac{3n}{2}}
\|u\|_{L^1(B_1^{+})}^2\\
& \quad \quad  \text{ for all } B_{\rho}^{+}(y)\subset B_1^{+} \text{ with } y \in \partial^0 B_1^{+} \text{ and } \delta \in (0,1).
\end{split}
\end{equation}
To deduce the desired bound, we must combine \eqref{dec1} with the following interior estimates derived in \cite{ErnetaInterior}*{Proposition~1.3}:
\begin{equation}
\label{dec2}
\rho^{n+2} \int_{B_{\rho/2}(y)} |\nabla u|^2 \d x \leq 
C  \|u\|_{L^1(B_1^{+})}^2
\quad \text{ for all } B_{\rho}(y) \subset B_1^{+}.
\end{equation}

We now claim that for all balls $B_{\rho}(y) \subset B_1$ (not necessarily contained in $B_1^{+}$) and every $\delta \in (0,1)$, we have
\begin{equation}
\label{dec3}
\begin{split}
\rho^{n+2}\int_{\partial \R^{n}_+ \cap B_{\rho/2}(y)}|\nabla u|^2\,dx
\leq & C \delta \rho^{n+2}\int_{\partial \R^{n}_+ \cap B_{\rho}(y)}|\nabla u|^2\,dx
+C\delta^{-2-\frac{3n}{2}} \|u\|_{L^1(B_1^{+})}^2.
\end{split}
\end{equation}
This is achieved by a simple covering argument. 
The key observation is that $\R^{n}_{+} \cap B_{\rho/2}(y)$ can be covered by a dimensional number of balls $\{B_{\rho/16}(y_i)\}_{i}$ and $\{B_{3 \rho/16}(z_j)\}_{j}$, where $y_i$ are such that $B_{\rho/8}(y_i) \subset \R^{n}_{+}\cap B_{\rho}(y) \subset B_1^{+}$ are interior balls,
while $z_j \in \partial \R^{n}_+$ satisfy $B_{3\rho/8}^{+}(z_j) \subset \R^{n}_{+} \cap B_{\rho}(y) \subset B_1^{+}$. 
Applying \eqref{dec2} to the interior balls and \eqref{dec1} to the boundary balls, it is not hard to deduce \eqref{dec3}.
For more details, we refer the reader to the proof of Lemma~8.2 in \cite{CabreQuant}.

By \eqref{dec3},
applying Simon's lemma
to the subadditive quantity $B \mapsto \|\nabla u\|_{L^2(\R^n_+ \cap B)}^2$
now yields the claim.
\end{proof}

Following ideas from~\cite{CabreFigalliRosSerra},
the higher integrability estimate in Theorem~\ref{thm:holder} will now be a direct consequence of the Hessian estimates in Theorem~\ref{thm:sz} and of Lemma~\ref{lemma:l2l1}.
%

\begin{proof}[Proof of Theorem~\ref{thm:holder}]

There are three steps in our proof.
First, 
by the divergence theorem and Theorem~\ref{thm:sz},
we control the surface integral of $|\nabla u|^2$ on every level set of $u$ 
by the Dirichlet energy.
Secondly, using coarea formula, H\"{o}lder, and Sobolev inequality, 
we will bound the $L^{2+\gamma}$ norm of the gradient by the $L^2$ norm.
Finally, Lemma~\ref{lemma:l2l1} will yield the final estimate in terms of the $L^1$ norm of the solution.
All these bounds are shown under a smallness condition on the coefficients which is removed in the last step.

\vspace{3mm}\noindent
\textbf{Step 1:}
{\it We prove that, if $\varepsilon \leq \varepsilon_0$, then for a.e. $t \in \R$ we have}
\[
\int_{\{u = t\} \cap B_{1/2}} |\nabla u|^2 \d \mathcal{H}^{n-1} \leq  
C  \|\nabla u\|_{L^2(B_{1})}^2,
\]
{\it where $\varepsilon_0 > 0$ and $C$ are universal.}

Since $\left|\div\big( |\nabla u| \nabla u\big) \right| \leq C |D^2 u|  |\nabla u|$, by \eqref{ineq:hessgrad} in Theorem~\ref{thm:sz},
for $\varepsilon \leq \varepsilon_0$
we have
\begin{equation}
\label{sob1}
\begin{split}
\left\|\div\big( |\nabla u| \nabla u\big)\right\|_{L^1(B_{4/7}^{+})} 
\leq C  \|\nabla u\|_{L^2(B_{1}^{+})}^2.
\end{split}
\end{equation}
Consider a cut-off function $\eta\in C^{\infty}_c(B_{4/7})$ with $\eta = 1$ in $B_{1/2}$ and $0 \leq \eta \leq 1$.
By the divergence theorem, for a.e. $t \in \R$ we have
\[\begin{split}
&\int_{\{u = t\} \cap B^{+}_{1/2}} |\nabla u|^2 \d \hcal^{n-1} \\
& \quad \quad  \leq \int_{\{u = t\} \cap B^{+}_{1} \cap \{\nabla u \neq 0\}}  |\nabla u|^2 \eta^2 \d \hcal^{n-1}\\
& \quad \quad = - \int_{\{u > t\} \cap B^{+}_{1} \cap \{\nabla u \neq 0\}} \div\big(|\nabla u| \nabla u \, \eta^2 \big) \d x 
- \int_{\{u > t\} \cap \partial^0 B_1^{+} \cap \{\nabla u \neq 0\}} |\nabla u|^2 \, \eta^2 \d x \\
& \quad \quad\leq \int_{B^{+}_{4/7}} |\nabla u|^2 |\nabla (\eta^2)| \d x + \int_{B_{4/7}^{+}} \big|\div\big(|\nabla u| \nabla u\big)\big| \eta^2 \d x
\end{split}\]
and \eqref{sob1} now yields the claim

\vspace{3mm}\noindent
\textbf{Step 2:}
{\it We prove that, if $\varepsilon \leq \varepsilon_0$, then}
\[
\|\nabla u\|_{L^{2+\gamma}(B_{1/2}^{+})} \leq C \| \nabla u\|_{L^2(B_1^{+})},
\]
{\it where $\gamma > 0$ is dimensional and
$\varepsilon_0 > 0$ and
$C$ are universal constants.}

Multiplying by a constant, we may assume that $\|\nabla u\|_{L^2(B_1^{+})} = 1$.

Letting $h(t) = \max\{1, t\}$, by the Sobolev embedding for functions vanishing on $\partial^0 B_1^{+}$, 
\begin{equation}
\label{sob2}
\begin{split}
&\int_{\R^+} \d t \int_{\{u = t\} \cap B_1^{+} \cap \{|\nabla u| \neq 0\}} \d \hcal^{n-1} h(t)^{p} |\nabla u|^{-1} \\
& \quad  \quad  \quad  \quad  \quad  \quad \leq |B_1^{+} \cap \{u < 1\}| + \int_{B_1^{+}} u^p \d x \leq C
\end{split}
\end{equation}
for some $p > 2$.
Choosing dimensional constants $q >1$ and $\theta \in (0, 1/3)$ such that $p/q = (1-\theta)/\theta$, we obtain
\[
\begin{split}
\int_{B_{1/2}^{+}} |\nabla u|^{3 - 3\theta} \d x &= \int_{\R^{+}} \d t \int_{\{u = t\} \cap B_{1/2}^{+} \cap \{|\nabla u| \neq 0\}} \d \hcal^{n-1} h(t)^{p\theta - q(1-\theta)} |\nabla u|^{-\theta + 2 (1 - \theta)}\\
&\leq 
\left(\int_{\R^+} \d t \int_{\{u = t\} \cap B_1^{+} \cap \{|\nabla u| \neq 0\}} \d \hcal^{n-1} h(t)^{p} |\nabla u|^{-1} \right)^{\theta}\\
& \quad  \quad  \quad  \quad \cdot \left( \int_{\R^{+}} h(t)^{-q} \d t \int_{\{u = t\} \cap B_{1/2}^{+}} \d \hcal^{n-1} |\nabla u|^2 \right)^{1-\theta}.
\end{split}
\]
By Step 1 and \eqref{sob2}, it follows that
\[
\int_{B_{1/2}^{+}} |\nabla u|^{3- 3\theta} \d x \leq C,
\]
which was the claim.

\vspace{3mm}\noindent
\textbf{Step 3:}
{\it Conclusion.}

Combining Step 2 (rescaled) and Lemma~\ref{lemma:l2l1}, we deduce that our class of stable solutions satifies
\begin{equation}
\label{sob3}
\|\nabla u\|_{L^{2+\gamma}(B_{1/4}^{+})} \leq C \|u\|_{L^1(B_1^{+})} \quad \text{ for } \varepsilon \leq \varepsilon_0,
\end{equation}
where $\gamma > 0$ is dimensional and $\varepsilon_0 > 0$ and $C$ are universal.

To conclude, we apply a simple covering argument.
Let $\delta \in (0,1)$ be sufficiently small such that 
\begin{equation}
\label{sob4}
\delta \left(\|DA \|_{L^{\infty}(B_1^+)} + \|\vv\|_{L^{\infty}(B_1^+)} \right) \leq \varepsilon_0.
\end{equation}
First, we cover the lower boundary $\partial^0 B_{1/2}^{+}$ by a finite number of balls $B_{\delta/4}(y_i)$ with $y_i \in \partial^0 B_1^+$,
taking $\delta > 0$ smaller if necessary so that $B_{\delta}(y_i) \subset B_{1}$.
Next, we cover $\overline{B_{1/2}^{+}} \setminus \left( \cup_{i} B_{\delta/4}(y_i) \right)$ by balls $B_{\widetilde{\delta}/2}(z_i)$ with a smaller radius $\widetilde{\delta} > 0$ such that $B_{\widetilde{\delta}}(z_i) \subset B_{1}^{+}$.
Thus we obtain a covering of $B_{1/2}^{+}$ by 
half-balls $\{B^{+}_{\delta/4}(y_i)\}_{i}$ (centered at the boundary)
and interior balls $\{B_{\widetilde{\delta}/2}(z_i)\}_{i}$,
satisfying $B^{+}_{\delta}(y_i) \subset B_{1}^{+}$ and $B_{\widetilde{\delta}}(z_i) \subset B_{1}^{+}$, respectively.
Notice that, by \eqref{sob4}, the radii $\delta$ and $\widetilde{\delta}$
as well as the number of balls depend only on $n$, 
$\varepsilon_0$, $\|D A\|_{L^{\infty}(B_1^+)}$, and $\|\vv\|_{L^{\infty}(B_1^+)}$.

Thanks to \eqref{sob4}, the function $u(y_i + \delta \cdot)$ vanishing on $\partial^0 B_1^+$ is a stable solution of a semilinear equation in $B_1^+$,
with coefficients $A^{y_i, \delta} = A(y_i + \delta \cdot)$ and $\vv^{y_i, \delta} = \delta \, \vv(y_i + \delta \cdot)$
such that $\|D A^{y_i, \delta}\|_{L^{\infty}} + \|\vv^{y_i, \delta}\|_{L^{\infty}} \leq \varepsilon_0$.
From \eqref{sob3} now we deduce
\begin{equation}
\label{sob5}
\|\nabla u\|_{L^{2+\gamma}(B_{\delta/4}^{+}(y_i))} \leq C_{\delta} \|u\|_{L^1(B_{\delta}^{+}(y_i))},
\end{equation}
where $C_{\delta}$ depends only on $n$, $\elliptic$, $\bounded$, and $\delta$.
For the interior balls $B_{\widetilde{\delta}/4}(z_i)$,
we need the following interior estimates from~\cite{ErnetaInterior}*{Theorem~1.1}:
\begin{equation}
\label{sob6}
\| \nabla u\|_{L^{2+\gamma}(B_{\widetilde{\delta}/2}(z_i))} \leq C_{\widetilde{\delta}} \|u\|_{L^1(B_{\widetilde{\delta}}(z_i))},
\end{equation}
where $C_{\widetilde{\delta}}$ depends only on $n$, $\elliptic$, $\bounded$, and $\widetilde{\delta}$.

By \eqref{sob5} and \eqref{sob6}, we finally obtain
\[
\begin{split}
\|\nabla u\|_{L^{2+\gamma}(B_{1/2}^{+})} & \leq 
\sum_{i} \| \nabla u \|_{L^{2+\gamma}(B^{+}_{\delta/4}(y_j))}
+\sum_{i} \|\nabla u\|_{L^{2+\gamma}(B_{\widetilde{\delta}/2}(z_i))}\\
& \leq C_{\delta} \sum_{i} \| u\|_{L^{1}(B^{+}_{\delta}(y_i))} 
+C_{\widetilde{\delta}} \sum_{i} \| u\|_{L^{1}(B_{\widetilde{\delta}}(z_i))} \\
& \leq C\|u\|_{L^1(B_1^+)},
\end{split}
\]
where the last constant depends only on $n$, $\elliptic$, $\bounded$,
$\|D A\|_{L^{\infty}(B_1^+)}$, and $\|\vv\|_{L^{\infty}(B_1^+)}$.
This concludes the proof of the theorem.
\end{proof}

\begin{remark}
\label{remark:gehring}
It is also possible to deduce a higher integrability of the gradient from Lemma~\ref{lemma:l2l1} directly by applying Gehring's lemma~\cite{Gehring}.
However, by that method, the integrability exponent 
in Theorem~\ref{thm:holder} would no longer be dimensional (i.e., depending only on $n$),
but would additionally depend on the ellipticity constants.\footnote{Indeed, combining Lemma~\ref{lemma:l2l1} with the analogous interior estimates in~\cite{ErnetaInterior}*{Proposition~1.3}, by Poincar\'{e}'s inequality and a scaling and covering argument, it is not hard to show that the (say) even reflection of $\nabla u$ with respect to $\{x_n = 0\}$ satisfies $\left(R^{-n} \int_{B_R(x)} |\nabla u|^2 \right)^{1/2} \leq C_1 R^{-n} \int_{B_{2R}(x)} |\nabla u|$ for any ball $B_{2R}(x) \subset B_1$, where $C_1 = C_1(n, \elliptic, \bounded)$ is a universal constant. Applying Gehring's lemma (for instance, by Theorem~6.38 in~\cite{GiaquintaMartinazzi}) we now obtain an estimate $\|\nabla u\|_{L^p(B_{1/2}^+)} \leq C \|\nabla u\|_{L^2(B_1^+)}$ for some $p = p(n, C_1) > 2$ and $C = C(n, C_1).$}
Thus, the techniques in~\cite{CabreFigalliRosSerra} give a more precise control of the integrability exponent than Gehring's lemma.
For instance, following the proof above, it is easy to see that one can take any $\gamma(n) < \frac{4}{3n-2}$.
\end{remark}

We conclude this section by stating a corollary of the higher integrability and Hessian estimates that will be useful in our next paper~\cite{ErnetaBdy2}.
It consists of two simple estimates on annuli that
can be proven by a standard covering argument,
combining Theorem~\ref{thm:holder} (respectively Theorem~\ref{thm:sz} and Lemma~\ref{lemma:l2l1})
with the analogous interior estimates in~\cite{ErnetaInterior}*{Theorem~1.1} (respectively in~\cite{ErnetaInterior}*{Proposition~1.3 \& Remark~3.4}).

\begin{corollary}
\label{cor:annuli}
Let $u \in W^{3,p}(B_1^+)$, for some $p > n$,
be a nonnegative stable solution of 
$- L u = f(u)$ in $B_{1}^{+}$, with $u = 0$ on $\partial^0 B_{1}^{+}$.
Assume that $f \in C^1(\R)$ is nonnegative and nondecreasing.
Assume that $L$ satisfies conditions~\eqref{elliptic},~\eqref{reg:coeffs}, and~\eqref{reg:b} in $\Omega = B_1^{+}$, and
\[
\|D A\|_{L^{\infty}(B_{1}^{+})} + \|\vv\|_{L^{\infty}(B_{1}^{+})} \leq \varepsilon
\]
for some $\varepsilon > 0$.
Let $0 < \rho_1 < \rho_2 < \rho_3 < \rho_4 \leq 1$.

Then
\[
\|\nabla u\|_{L^{2+\gamma}(A_{\rho_2, \rho_3}^{+})} \leq C_{\varepsilon, \rho_i} \|u\|_{L^1(A_{\rho_1, \rho_4}^{+})}
\]
and
\[
\|D^2 u \|_{L^1(A_{\rho_2, \rho_3}^{+})} \leq C_{\varepsilon, \rho_i} \|u\|_{L^1(A_{\rho_1, \rho_4}^{+})},
\]
where $C_{\varepsilon, \rho_i}$ is a constant depending only on $n$, $\elliptic$, $\bounded$, $\varepsilon$, $\rho_1$, $\rho_2$, $\rho_3$, and $\rho_4$.
\end{corollary}

\appendix

\medskip\medskip

\section{Two interpolation inequalities}
\label{app:interpolation}

We recall two interpolation inequalities in cubes by Cabr\'{e}~\cite{CabreRadial} (with elementary proofs in that paper).
In the first one, the $L^2$ norm of the gradient is bounded by a weighted $L^1$ norm of the Hessian and the $L^2$ norm of the function.
The second inequality controls this last integral by the $L^2$ norm of the gradient and the $L^1$ norm of the function.

\begin{proposition}[\cite{CabreRadial}]
\label{interpol}
Let $Q = (0,1)^n \subset \R^{n}$ and $u \in C^2(\overline{Q})$.

Then, for every $\delta \in (0,1)$,
\[
\| \nabla u \|_{L^2(Q)}^2 \leq C \left(  \delta \| \,|\nabla u| \, D^2 u \, \|_{L^1(Q)} + \delta^{-2} \|u\|_{L^2(Q)}^2 \right),
\]
where $C$ is a constant depending only on $n$.
\end{proposition}

\begin{proposition}[\cite{CabreRadial}]
\label{nash}
Let $Q = (0,1)^n \subset \R^{n}$ and $u \in C^2(\overline{Q})$.

Then, for every $\widetilde{\delta} \in (0,1)$,
\[
\| u\|_{L^2(Q)}^2 \leq C \left( \widetilde{\delta}^2 \|\nabla u\|_{L^2(Q)}^2 + \widetilde{\delta}^{-n} \|u\|_{L^1(Q)}^2 \right),
\]
where $C$ is a constant depending only on $n$.
\end{proposition}

%
%
%
%

\section{Absorbing errors in larger balls}
\label{app:simon}

We recall a celebrated device of Simon~\cite{Simon} for absorbing errors in large balls when controlling quantities in smaller balls:

 \begin{lemma}[\cite{Simon}]
Let $\beta\geq 0$ and $C_0>0$. 
Let  $\mathcal{B}$ be the class of all open balls $B$ contained in the unit ball $B_1$ of $\R^n$ and let 
$\sigma \colon \mathcal{B} \rightarrow [0,+\infty)$ satisfy the following subadditivity property:
\[
\sigma(B)\leq \sum_{j=1}^N \sigma(B^j) \quad \mbox{ whenever }  N\in \mathbb{Z}^+, \{B^j\}_{j=1}^N \subset \mathcal{B}, \text{ and } B \subset \bigcup_{j=1}^N B^j. 
\]

It follows that there exists a constant $\delta>0$, which depends only on $n$ and $\beta$, such that if
\[
\rho^\beta \sigma\left(B_{\rho/2}(y)\right) \leq \delta \rho^\beta \sigma\left(B_\rho(y)\right)+ C_0\quad \mbox{whenever }B_\rho(y)\subset B_1,
\]
then
\[ 
\sigma(B_{1/2}) \leq C C_0
\]
for some constant $C$ which depends only on $n$ and $\beta$.
\end{lemma} 

\section*{Acknowledgments}

The author wishes to thank Xavier Cabr\'{e} for useful discussions on the topic of this article, as well as for his encouragement over the years.



\end{document}